\documentclass[12pt]{amsart}

\usepackage[margin=1in]{geometry}
\usepackage[utf8]{inputenc}
\usepackage[english]{babel}
\usepackage{amssymb}
\usepackage{amsmath}
\usepackage{amsthm}
\usepackage{dsfont}
\usepackage{hyperref}
\usepackage{amsmath}
\usepackage{cite}
\usepackage{xcolor}
\hypersetup{
    colorlinks,
    linkcolor={blue!80!black},
    citecolor={blue!60!black},
    urlcolor={blue!80!black},
}

\usepackage [english]{babel}
\usepackage [autostyle, english = american]{csquotes}
\MakeOuterQuote{"}

\newtheorem{thm}{Theorem}[section]

\newtheorem{lem}[thm]{Lemma}
\newtheorem{prop}[thm]{Proposition}
\newtheorem{conj}[thm]{Conjecture}

\theoremstyle{definition}

\theoremstyle{definition}
\newtheorem{defi}[thm]{Definition}
\newtheorem*{definition-non}{Definition}

\theoremstyle{remark}
\newtheorem*{rmkx}{Remark}
\newenvironment{rmk}
  {\pushQED{\qed}\rmkx}
  {\popQED\endrmkx}

\newcommand{\eps}{\varepsilon}

\newcommand{\rn}{\mathbb R^n}

\newcommand{\sn}{S^{n-1}}

\newcommand{\kn}{\mathcal K^n}
\newcommand{\kno}{\mathcal K^n_o}

\newcommand{\po}{\mathcal P}

\newcommand{\fn}{\mathfrak{n}}

\newcommand{\bla}{\raise.2ex\hbox{$\scriptstyle\pmb \langle$}}
\newcommand{\sbla}{\raise.1ex\hbox{$\scriptscriptstyle\pmb \langle$}}
\newcommand{\bra}{\raise.2ex\hbox{$\scriptstyle\pmb \rangle$}}
\newcommand{\sbra}{\raise.1ex\hbox{$\scriptscriptstyle\pmb \rangle$}}

\newcommand{\balpha}{\pmb{\alpha}}

\numberwithin{equation}{section}

\begin{document}

\title{The Gauss Image Problem with weak Aleksandrov condition}

\author{Vadim Semenov}
\address{
Courant Institute of Mathematical Sciences, New York University, 251 Mercer St., New York, NY 10012, USA}
\curraddr{}
\email{vs1292@nyu.edu}

\subjclass[2010]{52A20, 52A38, 52A40, 52B11, 35J20, 35J96}

\keywords{Convex Geometry, The Gauss Image Problem, Aleksandrov Condition,  Monge-Amp\`ere equation, Aleksandrov Problem}

\date{\today}

\dedicatory{}


\begin{abstract}
We introduce a relaxation of the Aleksandrov condition for the Gauss Image
Problem. This weaker condition turns out to be a necessary condition for two measures to be related by a convex body. We provide several properties of the new condition. A solution to the Gauss Image Problem is obtained for the case when one of the measures is assumed to be discrete and the another measure is assumed to be absolutely continuous, under the new relaxed assumption. 
 \end{abstract}
 
 \maketitle

\tableofcontents
\newpage

\section{Introduction}
  The Gauss Image Problem, introduced in \cite{GIP}, is a natural extension to the classical Aleksandrov question of finding a body with the prescribed Aleksandrov's integral curvature \cite{Aleks1,Aleks0,Aleks}. This problem is a part of the study of Minkowski problems, a vital area of research in convex geometry. The study of these problems has led to the formulation of the log-Brunn-Minkowski conjecture \cite{L^p BM,log,KolesnikovMilman,Saroglou,Ramon} and to the sharp affine $L^p$ Sobolev inequality \cite{Sobolev}. The latter has also inspired many other sharp affine isoperimetric inequalities \cite{Haberl,LYZ00jdg,Sobolev}. Readers are referred to Chapters 8 and 9 of Schneider's textbook \cite{S14} for an introduction to Minkowski problems and to the articles \cite{BHP17jdg,BLYZ13jams,CW06adv,Lp Aleks,HLYZ16,HZ18adv,LutwakLp,LO95jdg,LYZ00jdg,LYZ04tams,LYZ06imrn,LYZ16,Ol2,Ol21,Oliker,Sta1,YZCVPDE,YZJDG,Zhao,Zu,Zu2} for an overview of the recent developments. Additionally, we acknowledge the works \cite{Cheng,Caffarelli,Nirenberg,Pogorelov,Trudinger} related to the regularity of Minkowski problems.
 
Given two measures $\mu$ and $\lambda$ on $\sn$, the Gauss Image Problem asks about the existence of a convex body $K$, containing the origin in its interior, such that $\mu=\lambda(K,\cdot)$, where by $\lambda(K,\cdot)$ we denote the pullback of $\lambda$ under the radial Gauss Image map of $K$: a composition of the multivalued Gauss map of $K$ and the radial map of $K$. More formally, given a Borel set $\omega\subset \sn$ we define:
\begin{equation}
	\lambda(K,\omega):=\lambda\left(\bigcup_{u\in\omega}N(K,\rho_K(u)u)\right),
\end{equation}
where $N(K,v)$ is the normal cone of a boundary point $v\in \text{bd}(K)$ and $\rho_K(\cdot)$ is the radial function of $K$.

 Many significant measures can be described as pullbacks of a certain $\lambda$ under the Gauss Image map. For instance, when $\lambda$ is the spherical Lebesgue measure, $\lambda(K,\cdot)$ is known as Aleksandrov's integral curvature of the body $K$ \cite{Aleks}. When $\lambda$ is Federer's $(n-1)^\text{th}$ curvature measure, $\lambda(K,\cdot)$ is the surface area measure of Aleksandrov-Fenchel-Jessen \cite{Aleks0}. Finally, the more recently defined dual curvature measure is also a pullback of a certain $\lambda$ under the Gauss Image map  \cite{HLYZ16}. All of these examples motivate the necessity for a systematic study of how measures transfer to each other through the radial Gauss Image Map, that is, the Gauss Image Problem:
 
 \smallskip
 
 \noindent{\textbf{The Gauss Image Problem}} (Raised in \cite{GIP}) \textit{Suppose $\lambda$ is a measure defined on the Lebesgue measurable subsets of $\sn$, and $\mu$ is a Borel measure on $\sn$. What are the necessary and sufficient conditions on $\lambda$ and $\mu$, so that there exists a convex body $K$ with the origin in its interior such that \begin{equation}
  \mu=\lambda(K,\cdot)?
\end{equation}
If such a convex body exists, to what extent is it unique?}  

 \smallskip
    
When $\lambda$ is a spherical Lebesgue measure, we recover the original Aleksandrov problem, which Aleksandrov first studied in\cite{Aleks1,Aleks,Aleks0}. Different proofs of the Aleksandrov problem were given by Oliker \cite{Oliker} and Bertrand \cite{Bertrand}. The $L_p$ analogs of the Aleksandrov problem were considered by Huang, Lutwak, Yang, and Zhang in \cite{Lp Aleks}, by Mui in \cite{Stephanie}, and by Zhao in \cite{Zhao}. 

When one of the measures is assumed to be absolutely continuous, the Gauss Image Problem was studied in \cite{GIP} by B\"or\"oczky, Lutwak, Yang, Zhang, and Zhao. There, the Aleksandrov relation was introduced to attack the problem: 
 
 \begin{defi}
 	Two Borel measures $\mu$ and $\lambda$ on $\sn$ are called Aleksandrov related if 
 	\begin{equation}
  \lambda(\sn)=\mu(\sn)>\mu(\omega)+\lambda(\omega^*)
\end{equation}
for each compact, spherically convex set $\omega\subset\sn$, where the set $\omega^*\subset\sn$ is defined as the spherical polar set to $\omega$: 
\begin{equation}
\omega^*:=\bigcap_{u\in\omega}\{v\in\sn : u\cdot v\leq 0\}	.
\end{equation}
\end{defi}
Equivalently, one can define two Borel measures $\mu$ and $\lambda$ on $\sn$ to be Aleksandrov related if $\mu(\sn)=\lambda(\sn)$ and for each compact, spherically convex set $\omega\subset\sn$, \begin{equation}
		\mu(\omega)<\lambda(\omega_{\frac\pi2}),
	\end{equation}
	where 
	\begin{equation}	
		\omega_{\frac\pi2}:=
 \bigcup_{u\in \omega} \{v\in \sn : u\cdot v > 0 \}.
	\end{equation}

 With this new condition, the following solution to the Gauss Image Problem was obtained:
\begin{thm}[K. J. B\"or\"oczky, E. Lutwak, D. Yang, G. Zhang and Y. Zhao \cite{GIP}]\label{GIP classical}
Suppose $\mu$ and $\lambda$ are Borel measures on $\sn$, and $\lambda$ is absolutely continuous. If $\mu$ and $\lambda$ are Aleksandrov related, then there exists a convex body $K$ containing the origin in its interior, such that $\mu=\lambda(K,\cdot)$.
\end{thm} 
Moreover, if the absolutely continuous measure $\lambda$ is strictly positive on open sets, it was shown that the Aleksandrov relation is a necessary assumption for the existence of a solution to the Gauss Image Problem. In this case, a solution to the Gauss Image Problem was shown to be unique up to a dilation. We refer the reader to \cite{GIP} for this result and an introduction to the Gauss Image Problem. Additionally, let us also mention Theorem 1.7 and Remark 4.9 in Bertrand \cite{Bertrand}, which also imply Theorem \ref{GIP classical} using a very different method. 

While the Aleksandrov relation is a natural assumption when one of the measures is assumed to be positive on open sets, it turns out that there are numerous examples of measures $\mu$ and $\lambda$ satisfying $\mu=\lambda(K,\cdot)$ that are not Aleksandrov related. For instance, let $K=B^{n}$, a unit ball centered at the origin, and $\mu$ and $\lambda$ to be any even, absolutely continuous, identical measures supported on small symmetric spherical caps $\omega$ and $-\omega$, where $\omega\subset\sn$ is a cap around the north pole and $-\omega\subset\sn$ is a cap around the south pole. Then, $\mu=\lambda=\lambda(K,\cdot)$, and $\mu(\omega)+\lambda(\omega^*)=\lambda(\sn)$, which violates the Aleksandrov relation. Moreover, starting with the body $K=B^{n}$, we can perturb it along the equator while preserving the convexity. We thereby obtain a family of convex bodies such that every member still solves the Gauss Image Problem for fixed measures $\mu$ and $\lambda$. This observation indicates that, in general, the solution to the Gauss Image Problem may be highly non-unique.

Based on these considerations, we introduce a relaxation of the Aleksandrov relation for the Gauss Image Problem. This relaxation turns out to be a necessary assumption for the two measures to be related by a convex body. That is, for the existence of a convex body $K$ with origin in its interior such that $\mu=\lambda(K,\cdot)$. See Proposition \ref{necessity}.

\begin{defi}
Given Borel measures $\mu$ and $\lambda$ on $\sn$, we say that $\mu$ is weakly Aleksandrov related to $\lambda$ if $\mu(\sn)=\lambda(\sn)$ and for each closed set $\omega\subset\sn$, contained in a closed hemisphere, there exists $\alpha\in(0,\frac \pi 2)$ such that
	\begin{equation}
		\mu(\omega)\leq\lambda(\omega_{\frac\pi2-\alpha}),
	\end{equation}
	where 
	\begin{equation}	
		\omega_{\frac\pi2-\alpha}:=
 \bigcup_{u\in \omega} \{v\in \sn : u\cdot v > \cos(\frac\pi 2 -\alpha) \}.
	\end{equation}
\end{defi}

Besides showing that the weak Aleksandrov relation is a necessary assumption for the existence of a solution to the Gauss Image Problem, we also show that the classical Aleksandrov relation implies the weak Aleksandrov relation. The nature of the constant $\alpha$ is addressed in Section \ref{Weak section}. In particular, if $\mu = \lambda(K,\cdot)$, then the constant $\alpha$ in the weak Aleksandrov condition is closely related to the inner to outer radius ratio of the body $K$. See Proposition \ref{necessity} and the discussion after it. 

Now, with an appropriate necessary condition, we are ready to state the main result of the paper. In the following, a measure $\mu$ is a discrete measure if it can be expressed as
\begin{equation}
	\mu=\sum\limits_{i=1}^m\mu_i\delta_{v_i}
\end{equation}
where $\mu_i$ are some positive coefficients, and $\delta_{v_i}$ is a Dirac measure of the set $\{v_i\}$. For the measure $\mu$, we also define the set $P_\mu$ of polytopes as \begin{equation}\label{convex hull notation}
	\po_\mu =\{\text{conv}\{\beta_iv_i\mid 1 \leq i \leq m\}\mid (\beta_1,\ldots,\beta_m)\in \mathbb R^m_{>0}\}.
\end{equation}

\begin{thm}\label{main}
Suppose $\mu$ and $\lambda$ are Borel measures on $\sn$, so that $\mu$ is discrete and not concentrated on a closed hemisphere, and $\lambda$ is absolutely continious. If $\mu$ is weakly Aleksandrov related to $\lambda$, then there exists a polytope $P\in \po_\mu$ such that $\mu=\lambda(P,\cdot)$. 
\end{thm}

In particular, we establish that given a discrete measure $\mu$ and an absolutely continuous measure $\lambda$, the weak Aleksandrov relation is a necessary and sufficient condition for the existence of a solution to the Gauss Image Problem.

In conclusion, we would like to comment on the differences between the methods introduced in this paper and those presented in \cite{GIP}. The proof of Theorem \ref{GIP classical} in \cite{GIP} has the following structure: First, it is shown that any convex body that maximizes the specific functional on convex bodies (defined below, see (\ref{specific functional})) is a solution to the Gauss Image Problem. Then, by analyzing the classical Aleksandrov relation for specific measures, it is proven that any sequence of convex bodies maximizing this functional exhibits a bound on the inner to outer radius ratio of its elements. This bound is arguably the most challenging aspect of the paper \cite{GIP}. From the Blaschke selection theorem, the authors then deduce that this sequence contains a convergent subsequence that converges to a non-degenerate convex body $K$ maximizing the functional (\ref{specific functional}). The limiting body $K$, in turn, solves the Gauss Image Problem.

The main challenge and difference in the proof of Theorem \ref{main}, as compared to the main result of \cite{GIP}, is that the weak Aleksandrov relation does not impose a bound on the inner to outer radius ratio for the possible solution, unlike its stronger counterpart. Going back to the previously mentioned example of spherical caps, for any scalars $\lambda_1,\lambda_2>0$, define $K_{\lambda_1,\lambda_2}$ to be the convex hull in $\mathbb R^n$ of  $\lambda_1\omega\subset\lambda_1\sn$ and $-\lambda_2\omega\subset\lambda_2\sn$. Note that any $K_{\lambda_1,\lambda_2}$ is a solution to $\mu=\lambda(K,\cdot)$, where $\mu$ and $\lambda$ are defined as before. (This is true because the normal cones of $K_{\lambda_1,\lambda_2}$  do not change for radial directions contained in the support of $\mu$, when we vary $\lambda_1$ and $\lambda_2$. See Section \ref{section preliminaries} for the definitions.) Hence, in contrast to the classical Aleksandrov relation assumption, the solution body may contain parts that one can dilate independently. Consequently, unlike the case when the classical Aleksandrov relation is assumed, a sequence of convex bodies, say $(K_\fn)_{\fn=1}^\infty$ such that $K_\fn\subset rB^n$ for all $\fn$ and some $r$, may maximize the functional while converging to a degenerate convex body. This makes the proof of the main theorem in this paper vastly differ from that in \cite{GIP}, as not every sequence of normalized convex bodies maximizing the functional is suitable for the proof. To construct this sequence and to overcome these challenges, we invoke a new process that we call the partial rescaling of convex bodies. See \eqref{convex hull form of pt}.

It would be very interesting to see whether one could prove the result of the Theorem \ref{GIP classical}, the main results of \cite{GIP}, under the weak Aleksandrov relation assumption instead of the classical Aleksandrov relation. Our paper can be viewed as a step towards this direction. We state this in the Conjecture \ref{GIP conj}.  

\vspace{4mm}
\textbf{Acknowledgements} The author is extremely grateful to the editor and the referee for their valuable and extensive comments, which significantly improved readability of the paper.

\section{Preliminaries}\label{section preliminaries}
By $\kn$ we denote the set of convex bodies (compact, convex subsets with nonempty interior in $\rn$). By $\kno\subset\kn$ we denote those convex bodies that contain the origin in their interiors. Given $K\in\kno$, let $x\in\partial K$ be a boundary point. \textit{The normal cone} at $x$ is defined by 
\begin{equation}
  N(K,x)=\{v\in\sn : (y-x)\cdot v\leq 0 \text{ for all } y\in K \},
\end{equation}
which parametrizes all unit normals at a given boundary point. For $K\in\kno$,
 \textit{the radial map} $r_K:\sn\rightarrow\partial K$ of $K$ is defined for $u\in \sn$ by $r_K(u)=ru\in\partial K$, where $r>0$. Given a subset $\omega$ of $\sn$, the \textit{radial Gauss Image} of $\omega$ is defined as follows:
\begin{equation}
\balpha_K(\omega)=\bigcup_{x\in r_K(\omega)}N(K,x)\subset \sn.	
\end{equation}
The radial Gauss Image map, $\balpha_K$, maps sets of $\sn$ to sets of $\sn$. Outside of a spherical set of Lebesgue measure zero, the multivalued map $\balpha_K$ is singular valued. It is known that $\balpha_K$ maps Borel measurable sets to Lebesgue measurable sets. See p.88--89 in \cite{S14} for both of these results.  We denote the restriction of $\balpha_K$ to the corresponding singular valued map by $\alpha_K$. For additional details, we refer the reader to \cite{GIP}. 

\textit{The radial function} $\rho_K:\sn\rightarrow \mathbb R$ is defined by: \begin{equation}
  \rho_K(u)=\max\{a : au\in K\}.
\end{equation}
In this case, $r_K(u)=\rho_K(u)u$. \textit{The support function} of the body $K$ is defined by: \begin{equation}
  h_K(x)=\max\{x\cdot y : y\in K\}.
\end{equation}
 For $K\in\kno$, we define its \textit{polar body} $K^*\in\kno$ as the convex body with support function given by $h_{K^*}:=\frac 1 {\rho_K}$. 
 
 We denote by $\mathfrak r_K$ the radius of the largest ball contained in $K$ and centered at $o$. Similarly, we denote $\mathfrak R_K$ to be the radius of the smallest ball containing $K$ and centered at $o$. We will refer to $\mathfrak r_K$ as \textit{the inner radius} of the body $K$, to $\mathfrak R_K$ as \textit{the outer radius} of $K$, and to the ratio $\frac {\mathfrak r_K} {\mathfrak R_k}$ as \textit{the inner to outer radius ratio} of the body $K$.
 
 It is important to note that for any $K\in\kno$, the following identity holds: \begin{equation}\label{some geometric obvious relation}
 	\min{\rho_K}=\min{h_K}=\mathfrak r_K\leq \mathfrak R_k=\max{\rho_K}=\max{h_K}. 
 \end{equation} \textit{The support hyperplane} to $K$ with an outer unit normal $v\in\sn$ is defined as \begin{equation}
 H_K(v)=\{x:x\cdot v=h_K(v)\}.	
 \end{equation}
By $H^-(\alpha,v)$ we denote the halfspace $\{x:x\cdot v\leq\alpha\}$ and by $H(\alpha,v)$ we denote the hyperplane $\{x:x\cdot v=\alpha\}$. Given a set $S\subset  \rn$ we write its convex hull as
\begin{equation}
	\text{conv}(S).
\end{equation}

For a set $\omega \subset S^{n-1}$, we define $\text{cone}\,\omega $ as
the {\it cone that $\omega $ generates} in $\rn$, that is
\begin{equation}
  \text{cone}\,\omega = \{tu:\text{$t\ge 0$ and $u\in\omega$}\}.
\end{equation}
We say that $\omega \subset S^{n-1}$ is {\it spherically convex} if the
cone that $\omega $ generates is a nonempty, proper, convex subset of $\rn$.
Therefore,
 a spherically convex set in $\sn$ is always nonempty and
 contained in a closed hemisphere of $\sn$. Given $\omega\subset S^{n-1}$ contained in a closed hemisphere,
 {\it the polar set} $\omega^*$ is defined by:
\begin{equation}\label{polarset}
\begin{split}	
\omega^* &=
\bigcap_{u\in\omega}\{v\in S^{n-1} : u\cdot v\leq 0 \}. 
\end{split}
\end{equation}
We note that the polar set is always spherically convex. If $\omega\subset\sn$ is a closed set, we define its {\it outer parallel set} $\omega_\alpha$ for some $\alpha\in (0,\frac \pi 2]$ to be
\begin{equation}
	\omega_\alpha =
 \bigcup_{u\in \omega} \{v\in \sn : u\cdot v > \cos\alpha \}.
\end{equation}
For notational convenience, if $\omega$ contains a single vector, say, $\omega=\{v\}$ where $v\in\sn$, we are going to simply write $v_\alpha$ instead of $\{v\}_{\alpha}$.

As mentioned previously, $\balpha_K$ maps Borel measurable sets to Lebesgue measurable sets. Given a Borel measure $\lambda$ we, as in \cite{GIP}, define \textit{the Gauss Image measure of }$\lambda$ \textit{via} $K$ as
\begin{equation}
	\lambda(K,\omega):=\lambda(\balpha_K(\omega))
\end{equation}
for each Borel $\omega\in\sn$. Note, however, that the naming is a bit misleading as, in general, $\lambda(K,\cdot)$ does not necessarily have to be a measure. For example, in the dimension 2, let $K$ be a square centered at the origin with sides perpendicular to unit vectors $u_1,u_2,u_3,u_4$. Let $\lambda=\sum\limits_{i=1}^4\delta_{u_i}$ where $\delta_{u_i}$ are Dirac measures of sets $\{u_i\}$. Let unit vectors $v_1$ and $v_2$ be such that $r_K(v_1),r_K(v_2)$ are in the interior of the side of $K$ perpendicular to $u_1$. Then, \begin{equation}
	\balpha_K(\{v_1\})=\balpha_K(\{v_2\})=\balpha_K(\{v_1,v_2\})=\{u_1\}.
\end{equation}
Implying that:
\begin{equation}
	1=\lambda(K,\{v_1\})=\lambda(K,\{v_2\})=\lambda(K,\{v_1,v_2\}),
\end{equation}
which establishes that $\lambda(K,\cdot)$ is not countably additive.

On the other hand, if $\lambda$ is an absolutely continuous Borel measure, which is the case of this work, $\lambda(K,\cdot)$ is always a measure. For this and related results, see \cite{GIP}. We also point out Lemma 3.3 in \cite{GIP}, which states that:

\begin{lem}
If $\lambda$ is an absolutely continuous Borel measure, and $K\in\kno$, then \begin{equation}
 	\int_{\sn}f(u)d\lambda(K,\cdot)=\int_{\sn}f(\alpha_K(v))d\lambda(v)
 \end{equation}	
 for each bounded Borel measurable function $f:\sn\rightarrow \mathbb R$.
\end{lem}
We note that if for a pair of a given $\mu$ and $\lambda$, there exists $K\in\kno$ such that $\mu=\lambda(K,\cdot)$, then we say that the measures $\mu$ and $\lambda$ are {\it related by the convex body} $K$. For $K\in\kno$ and $\lambda$ absolutely continuous, we define the functional $\Phi(K,\mu,\lambda)$ by 
\begin{equation}\label{specific functional}
	\Phi(K,\mu,\lambda):=\int_{\sn}\log\rho_Kd\mu+\int_{\sn}\log\rho_{K^*}d\lambda.
\end{equation}
Sometimes, we will write $\Phi(K)$, suppressing $\mu$ and $\lambda$. Note that $\Phi(K,\mu,\lambda)$ corresponds to $\Phi_{\mu,\lambda}(K^*)$ in the notation of \cite{GIP}. This functional is intimately associated with the Gauss Image Problem. For example, Theorem 8.2 in \cite{GIP} shows that if $\mu$ is a Borel measure and $\lambda$ is an absolutely continuous Borel measure such that\begin{equation}\label{maximizer its the solution}
	\Phi(K,\mu,\lambda)=\sup_{K'\in \kno} \Phi(K',\mu,\lambda)
\end{equation} for $K\in\kno$, then $\mu=\lambda(K,\cdot)$. It is important to stress that:
\begin{equation}\label{scaling motivation}
\begin{split}
	&\text{If } \mu=\lambda(K,\cdot) \text{, then } \mu=\lambda(cK,\cdot) \text{ for any }c>0.\\
	&\Phi(K,\mu,\lambda)=\Phi(cK,\mu,\lambda) \text{ for any } c>0.
\end{split}
\end{equation}
That is, the nature of the problem is not sensitive to the rescaling of the convex bodies. 

The Aleksandrov relation, as well as the weak Aleksandrov relation, were defined in the Introduction. We simply note the interchangeable use of the terms "Aleksandrov condition" and "Aleksandrov relation". 
 
A measure $\mu$ is called discrete if it takes the form: 
\begin{equation}\label{form of measure}
	\mu=\sum\limits_{i=1}^m\mu_i\delta_{v_i}
\end{equation} 
where $\delta_{v_i}$ are Dirac measures of sets $\{v_i\}$ containing a single vector $v_i\in\sn$ and $\mu_i$ are strictly positive coefficients. Aside from Proposition \ref{where mu is not discrete}, the measure $\mu$ will always be assumed to be discrete and written as in \eqref{form of measure} with letters $v$ and $m$ reserved specifically for $\mu$. 

Given a discrete measure $\mu$ not concentrated on a closed hemisphere, we define $\po_\mu$ to be the set of the convex hull of points $\{\beta_iv_i\}$ with $\beta_i>0$. See \eqref{convex hull notation}  . Given any $P\in \po_\mu$, since $\mu$ is not concentrated on a closed hemisphere, $P$ contains the origin in its interior. Therefore, $\po_\mu\subset \kno$. Moreover, any $P\in\po_\mu$ is a polytope, such that each vertex of $P$ is located in a radial direction $v_i$ for some $i\in\{1,\ldots,m\}$. Note, however, that sometimes a polytope $P\in\po_\mu$ might have fewer than $m$ vertices corresponding to some $\beta_jv_j$ contained inside the convex hull of the remaining points. 

The next part of notations can be viewed as a discrete analog to the standard concepts of the support function and the Wulff shape. Given $P\in \po_\mu$, we define its representation to be an $m$-tuple of positive numbers \begin{equation}\label{definition of representation}
	\alpha=(\alpha_1,\ldots,\alpha_m):=(h_{P^*}(v_1),\ldots,h_{P^*}(v_m)).
\end{equation} 
Note that if $\alpha$ is the representation of $P$ then 
\begin{equation}\begin{split}\label{Form equation}
	P&=\text{conv}\{ \frac {v_i} {\alpha_i} \mid 1\leq i \leq m \}, \\
 	P^*&=\bigcap_{i=1}^mH^-(\alpha_i,v_i).
\end{split}
\end{equation}

Conversely, suppose we start with some $m-$tuple of positive numbers, $\gamma$. We define $P_\gamma\in \po_\mu$ to be the following:
\begin{equation}
  P_\gamma=(\bigcap_{i=1}^mH^-(\gamma_i,v_i))^*. 
\end{equation}
We call such $P$ a \textit{dual Wulff Shape} of the $m-$tuple $\gamma$. We refer to the representation of $P_\gamma$ as the \textit{Wulff tuple} of $\gamma$. In particular, one has that if $\alpha$ is a Wulff tuple of $\gamma$, then 
\begin{equation}\label{feb13}
	\alpha_i\leq \gamma_i.
\end{equation}
Moreover, if $\alpha_i<\gamma_i$, then the facet of $P^*_\gamma$ in the direction $v_i$ is degenerate. 

If a polytope has an index $a$, such as $P_a$, we are going to write its coefficients in representation as: 
\begin{equation}
\begin{split}
	\alpha_a&=(\alpha_{a,1},\ldots,\alpha_{a,m}) 
\end{split}
\end{equation}

Given $P\in\po_\mu$ with its representation denoted by the $m$-tuple $\alpha$ and a nonempty and not full index set $I\subset \{1,\ldots, m\}$, we will denote by $U(I),L(I),U^*(I),L^*(I)$ the following quantities: \begin{equation}\label{notation UL 2}\begin{split}
  U(I):=&\max_{i\in I}\alpha_i, \\ L(I):=&\min_{i\in I}\alpha_i,\\ U^*(I):=&\max_{i\notin I}\alpha_i,\\ L^*(I):=&\min_{i\notin I}\alpha_i.
  \end{split}
\end{equation}
It will usually be the case that:
\begin{equation}
	0<L^*(I)\leq U^*(I)\leq L(I)\leq U(I)=1.
\end{equation} 
Outside Proposition \ref{maximizer} and Proposition \ref{weak Aleks bound}, the index set $I$ remains the same through each proposition. In this case we simply write \begin{equation}U,L,U^*,L^*\end{equation}  to denote the same quantities. 
If a polytope has an index $t$, such as $P_t$, we are going to write \begin{equation}
	P_t, \alpha_t, L_t,U_t,L_t^*,U_t^*.
\end{equation} to denote the same quantities for $P_t$.

We use the books of Schneider \cite{S14} as our standard reference. The books of Gruber and Gardner are also good alternatives \cite{G06book,Gruberbook}.

\section{Weak Aleksandrov Condition}\label{Weak section}
Let us start by showing that the weak Aleksandrov relation is a necessary condition for Borel measures to be related by a convex body.

\begin{prop}\label{necessity}
	Given $K\in\kno$, suppose $\lambda$ and $\lambda(K,\cdot)$ are Borel measures. Then, $\lambda$ is weakly Aleksandrov related to $\lambda(K,\cdot)$. 
\end{prop}
\begin{rmk}
	Note that if $\lambda$ is an absolutely continuous Borel measure, then $\lambda(K,\cdot)$ automatically becomes a Borel measure. For more details, see Section \ref{section preliminaries} and Lemma 3.3 in \cite{GIP}.
\end{rmk}
\begin{proof}
	Since $K\in\kno$, there exists $c>0$ such that $\frac{\mathfrak r_K}{\mathfrak R_K}>c$. Consider some $u\in\sn$ and $v\in\balpha_K(u)$. Then from \eqref{some geometric obvious relation},
	\begin{equation}
	\mathfrak r_K\leq h_K(v)= \rho_K(u)u\cdot v \leq \mathfrak R_Ku\cdot v. 	
	\end{equation}
Hence, $c<\frac{\mathfrak r_K}{\mathfrak R_K}\leq u \cdot v$. Therefore, for each $u\in\sn$, we have: \begin{equation}
  \balpha_K(u)\subset u_{ \arccos(c)}\subset u_{\frac \pi 2 - \alpha} 
  \end{equation}
  for some $\alpha$, where $0<\alpha<\frac \pi 2$. Therefore, for any closed set $\omega$ contained in a closed hemisphere, since $\omega_{\frac \pi 2 - \alpha}=\bigcup_{u\in \omega} u_{\frac \pi 2 - \alpha}$, we obtain:  \begin{equation}
  \lambda(K,\omega)=\lambda(\balpha_K(\omega))\leq\lambda(\omega_{\frac \pi 2 - \alpha}).
\end{equation}
\end{proof}
In particular, the above proof shows that given measures $\lambda$ and $\lambda(K,\cdot)$, the constant $\alpha$ in the weak Aleksandrov relation does not depend on the choice of a closed set $\omega$ contained in a closed hemisphere. Moreover, the constant $\alpha$ in the above proof encompasses the lower bound on $\frac {\mathfrak r_K} {\mathfrak R_K}$. That is, the bound on the inner to outer radius ratio of the body $K$. The following Proposition is the step in the opposite direction. Recall the notation for a discrete measure $\mu$ in \eqref{form of measure}.

\begin{prop}\label{uniform constant}
Suppose that a discrete Borel measure $\mu$ is weakly Aleksandrov related to a Borel measure $\lambda$. Then, there exists a uniform constant $\alpha\in (0,\frac \pi 2)$, such that for any a closed set contained in a closed hemisphere $\omega\subset\sn$: \begin{equation}
  \mu(\omega)\leq\lambda(\omega_{\frac\pi2-\alpha}).
\end{equation}
\end{prop}
\begin{rmk}
We refer to this $\alpha$ as \textit{the uniform weak Aleksandrov constant} for measures $\mu$ and $\lambda$.	
\end{rmk}

\begin{proof}
	Consider all possible $I\subset\{1,\ldots, m\}$ such that $\{v_i\}_{i\in I}$ are contained in a closed hemisphere. Let $\omega^I=\bigcup_{i\in I}{v_i}$. Since $\mu$ and $\lambda$ are weak Aleksandrov related for each $\omega^I$, we have $\mu(\omega^I)\leq\lambda(\omega^I_{\frac \pi 2-\alpha_I})$ for some $\alpha_I$. Since there are only finitely many of those $I$ satisfying the assumption, we can choose $\alpha>0$ to be the minimum of the $\alpha_I$'s. Now for any closed set $\omega$ contained in a closed hemisphere, we obtain that for some $I\subset\{1,\ldots, m\}$: 
\begin{equation}
  \mu(\omega)=\mu(\omega\cap\{v_i\}_{i\in \{1,\ldots, m\}})=\mu(\omega^I)\leq\lambda(\omega^I_{\frac \pi 2 - \alpha})\leq\lambda(\omega_{\frac \pi 2 - \alpha}),
\end{equation}
where the last step follows from set inclusion. 
\end{proof}

Finally, we note that the classical Aleksandrov relation easily implies the weak Aleksandrov relation. In the following, $\mu$ is not necessarily a discrete measure.

\begin{prop}\label{where mu is not discrete}
Suppose that a Borel measure $\mu$ is Aleksandrov related to a Borel measure $\lambda$. Then, $\mu$ is weakly Aleksandrov related to $\lambda$.
\end{prop}
\begin{proof}
	Since $\mu$ is Aleksandrov related to $\lambda$, for each compact, spherically convex set $\omega\subset\sn$, we obtain: \begin{equation}
  \mu(\omega)<\lambda(\sn)-\lambda(\omega^*)=\lambda(\omega_{\frac \pi 2}).
\end{equation}
Now, consider any closed set $\gamma$ contained in a closed hemisphere. Let \begin{equation}
  \omega=\bla\gamma\bra:=\sn\cap\text{conv}\,(\text{cone}\,\gamma), 
\end{equation}
where $\text{conv}\,(\text{cone}\,\gamma)$ denotes the convex hull of $\text{cone}\,\gamma$.
Note that the convex hull of any set $S\in\rn$ is given by all finite convex combinations of elements in $S$. Thus, recalling the definition of a cone, we obtain the following: for each $v\in \omega$, there exist vectors $\{v_i\}_{i\in I}\subset \gamma$ with $I$ finite index set such that
\begin{equation}
  v=\sum_{i\in I}\sigma_iv_i \text{ for some } \sigma_i>0.
\end{equation}

First we want to establish that $\omega_{\frac \pi 2}\subset \gamma_{\frac \pi 2}$. Choose any $u\in \omega_{\frac \pi 2}$. Then for some $v\in \omega$, $u \cdot v>0$. Hence, for some $\{v_i\}_{i\in I}\subset \gamma$ with $I$ finite index set,
\begin{equation}
  u \cdot v=\sum_{i\in I}\sigma_i(u \cdot v_i)>0 \text{, where } \sigma_i>0.
\end{equation}
Thus, at least for one $i\in I$, we have that $u \cdot v_i>0$. Hence, $u\in {v_i}_{\frac \pi 2}\subset \gamma_{\frac \pi 2}$. We obtained the desired. 

Combining, we obtain the following chain of inequalities: 
\begin{equation}
  \mu(\gamma)\leq\mu(\omega)<\lambda(\omega_{\frac \pi 2})\leq\lambda(\gamma_{\frac \pi 2}).
\end{equation}
In particular, obtaining strict inequality $\mu(\gamma)<\lambda(\gamma_{\frac \pi 2}).$
By the continuity of measure $\lambda$, $\lambda(\gamma_{\frac \pi 2-\alpha})\rightarrow \lambda(\gamma_\frac \pi 2)$ as $\alpha\rightarrow 0$. Hence, for a given closed set $\gamma$ contained in a closed hemisphere, there exists an $\alpha$ such that 
\begin{equation}
	\mu(\gamma)<\lambda(\gamma_{\frac \pi 2-\alpha}).
\end{equation}
The weak Aleksandrov condition follows.
\end{proof}

One might wonder whether we can define the weak Aleksandrov relation merely by restricting the definition to the collection of compact spherically convex sets instead of closed sets contained in a closed hemisphere. We have not investigated this question, leaving it to the reader if they are interested.

\section{Essential Estimates and Partial Rescaling of a Polytope}

For the rest of the paper, we will assume that $\mu$ is a discrete Borel measure written as in \eqref{form of measure}, which is not concentrated on a closed hemisphere. We will also assume that $\lambda$ is an absolutely continuous measure. We begin with the following lemma, which enables us to concentrate our attention exclusively on polytopes.

\begin{lem}\label{assume to be polytope}
	Given any $K\in\kno$, there exists a polytope $P\in\po_\mu$ such that: \begin{equation}
\Phi(P,\mu,\lambda)\geq\Phi(K,\mu,\lambda).
\end{equation}

\end{lem}
\begin{proof}
	Choose any $K\in\kno$. Define $P$ as:
	\begin{equation}\label{4.1.1}
		P:=\text{conv}\{\rho_K(v_i)v_i\mid 1\leq i \leq m\}.
	\end{equation}
Clearly, $P\in \po_\mu$. Moreover, since $P$ is a convex hull of a subset of $K$, $P\subset K$. Hence, $h_P\leq h_K$, which implies that: \begin{equation}
	\int_{\sn}\log\rho_{K^*}d\lambda\leq\int_{\sn}\log\rho_{P^*}d\lambda.
\end{equation}
Simultaneously, by the definition of $P$, for any $i$ we have that $\rho_P(v_i)\geq \rho_K(v_i)$. Since $P\subset K$, we also have that $\rho_P(v_i)\leq \rho_K(v_i)$. Therefore, for all $i$, $\rho_P(v_i)= \rho_K(v_i)$, and, thus, 

\begin{equation}\label{4.1.2}
	\int_{\sn}\log\rho_Kd\mu=\int_{\sn}\log\rho_Pd\mu.
\end{equation}
Combining both equations \eqref{4.1.1} and \eqref{4.1.2}, we obtain that $\Phi(K,\mu,\lambda)\leq\Phi(P,\mu,\lambda)$. 
\end{proof}

Theorem 8.2 in \cite{GIP} shows that if $K\in\kno$ maximizes the functional $\Phi(\cdot,\mu,\lambda)$ for a Borel measure $\mu$ and an absolutely continuous measure $\lambda$, then $K$ solves the Gauss Image Problem. Thus, to prove Theorem \ref{main}, it is sufficient to establish the existence of a $K\in\kno$ that maximizes the functional. Lemma \ref{assume to be polytope} permits us to restrict bodies to polytopes of the above form, allowing us to work exclusively within the class $P_\mu$.

We start with some lemmas concerning the class $P_\mu$. For the rest of the article, we are going to work with the notation defined in \eqref{convex hull notation}-\eqref{notation UL 2}. Recall that for a given $P\in \po_\mu$, we call an $m$-tuple $\alpha$ to be a representation of $P$ if 
\begin{equation}
		\alpha=(\alpha_1,\ldots,\alpha_m)=(h_{P^*}(v_1),\ldots,h_{P^*}(v_m)).
	\end{equation}

\begin{lem}\label{technical estimate 1}
	Given $P\in \po_\mu$, let $\alpha$ be its representation. Then, for any $u\in \sn$: \begin{equation}\label{min rho}
  \rho_{P^*}(u)=\min\left\{\frac{\alpha_i}{u\cdot v_i}\large \,\middle\vert\, i\in\{1,\ldots,m\}, u\cdot v_i>0 \right\}.
\end{equation}
Moreover, given $u\in\sn$
\begin{equation}
	\rho_{P^*}(u)=\frac{\alpha_i}{u \cdot v_i}
\end{equation}
if and only if for some $i$, $r_{P^*}(u)\in H_{P^*}(v_i)$.
\end{lem}
\begin{proof}
Recall from \eqref{Form equation} that $P^*$ can be written as,
\begin{equation}\label{form in 4.2}
	P^*=\bigcap_{i=1}^mH^-(\alpha_i,v_i).
\end{equation}
Fix some $u\in \sn$. Suppose for some $i$, $u\cdot v_i\leq 0$. Then $H^-(\alpha_i,v_i) $ contains the entire ray in the direction of $u$ starting at the origin. Suppose now for a given $i$, $u\cdot v_i> 0$. Then a ray in the direction of $u$ starting at the origin intersects the hyperplane $H(\alpha_i,v_i)$. In this case, it is straightforward to verify that the distance from the origin to the intersection will be: \begin{equation}
  \frac{\alpha_i}{u\cdot v_i}.
\end{equation}
Therefore, looking back at the equation \eqref{form in 4.2}, we obtain that  \begin{equation}\label{min rho}
  \rho_{P^*}(u)=\min\left\{\frac{\alpha_i}{u\cdot v_i}\large \,\middle\vert\, i\in\{1,\ldots,m\}, u\cdot v_i>0 \right\}.
\end{equation}
The last part of the statement follows from equation \eqref{min rho} and \eqref{form in 4.2}.
\end{proof}

The following lemma is a core estimate, which will later be used to properly rescale the polytopes without decreasing the value of the functional.

\begin{lem}\label{4.14}
Given $P\in \po_\mu$, let $\alpha$ be its representation. If we are given a nonempty and not full index set $I\subset \{1,\ldots, m\}$, with $U^*<L$, then: \begin{equation}
 \bigcup_{i\notin I}({v_i})_{\arccos\frac {U^*} L}\subset \balpha_P(\bigcup_{i\notin I}v_i).
\end{equation}
\end{lem}
\begin{proof}

 Suppose that for a given direction $u\in\sn$, we have $\rho_{P^*}(u)<L$. Then, using Lemma \ref{technical estimate 1} we obtain
 
 \begin{equation}\label{4.3.1}
\min\left\{\frac{\alpha_i}{u\cdot v_i}\large \,\middle\vert\, i\in\{1,\ldots,m\}, u \cdot v_i>0 \right\}=\rho_{P^*}(u)<L.
\end{equation}
Suppose the minimum is achieved for some index $j$. Then, from the previous equation, using the definition of $L$, we obtain: 
\begin{equation}
	\frac{\alpha_j}{u\cdot v_j}<L=\min_{i\in I}\alpha_i. 
\end{equation} 
Thus, since $0<u\cdot v_j\leq 1$, we obtain that $\alpha_j<\min_{i\in I}\alpha_i$, and, hence, $j\notin I$. Now, applying the second part of Lemma \ref{technical estimate 1} we obtain that \begin{equation}
\begin{split}
\rho_{P^*}(u)=\frac{\alpha_j}{u\cdot v_j} \Leftrightarrow\\
	r_{P^*}(u)\in H_{P^*}(v_j) \Leftrightarrow \\
	r_{P}(v_j)\in H_{P}(u)  \Leftrightarrow \\
	u\in \balpha_P(v_j).
\end{split}
\end{equation}
Since $j\notin I$, we obtain $u\in\balpha_P(\bigcup_{i\notin I}v_i)$. 

Thus, we have established that if for a given direction $u\in\sn$, we have $\rho_{P*}(u)<L$, then $u\in\balpha_P(\bigcup_{i\notin I}v_i)$. Now, pick any $u\in ({v_j})_{\arccos\frac {U^*} L}$ for some $j\notin I$. We obtain $u\cdot v_j>\frac {U^*} L$. Thus, combining this with equation \eqref{4.3.1} and the fact that $j\notin I$, we obtain:

 \begin{equation}
  \rho_{P^*}(u)=\min\left\{\frac{\alpha_i}{u \cdot v_i}\large \,\middle\vert\, i\in\{1,\ldots,m\}, u\cdot v_i>0 \right\}\leq \frac{\alpha_j}{u\cdot v_j}\leq\frac{U^*}{u\cdot v_j}< L. 
\end{equation}
Thus, $\rho_{P*}(u)<L$ and the claim follows from the first part of the proof.

\end{proof}
 
In the upcoming proof of Theorem \ref{main}, we will utilize what we refer to as \textit{the partial rescaling of a polytope}. Let us now describe this construction. Suppose we are given a polytope $P\in\po_\mu$, along with a nonempty and not full index set $I\in \{1,\ldots, m\}$. Let $\alpha$ denote the representation of $P$. Recall that $P$ and $P^*$ can be written as 
\begin{equation}\label{kasdjfkajf}\begin{split}
	P&=\text{conv}\{ \frac {v_i} {\alpha_i} \mid 1\leq i \leq m \}, \\
 	P^*&=\bigcap_{i=1}^mH^-(\alpha_i,v_i).
\end{split}
\end{equation}

We would like to rescale the half spaces that correspond to the index set $I$ in the second formula by a factor $t>0$. This procedure will be called a \textit{partial rescaling of polytope} $P$ \textit{with respect to index set} $I$. In terms of the preceding formula, the partial rescaling of a polytope $P$ can be written as
\begin{equation}
\begin{split}\label{convex hull form of pt}
	P_t&=\text{conv}\big(\{ \frac {v_i} {\alpha_i} \mid i\in I\}\cup \{ \frac {v_i} {t\alpha_i} \mid i\notin I\} \big),  \\
 	P^*_t&=\bigcap_{i\in I}H^-(\alpha_i,v_i)\bigcap_{i\notin I}H^-(t\alpha_i,v_i).
\end{split}
\end{equation}

In fact, $P_t$ is the dual Wulff shape of the $m$-tuple $\gamma_t$ defined by: \begin{equation}
\begin{split}
	\gamma_{t,i}&=\alpha_i \text{ if } i\in I, \\ 
	\gamma_{t,i}&=t\alpha_i  \text{ if } i\notin I.
\end{split}
\end{equation}
See the Preliminaries section for the definitions. We will always assume that $t\in(0,1]$. The most important point to make about the partial rescaling is that the representation of $P_t$ is not necessarily equal to $\gamma_t$. For example, if the set $\{v_i\mid i\notin I\}$ is not contained in a closed hemisphere, $P^*_t$ will approach the origin in the Hausdorff distance as $t\rightarrow 0$, and, thus, for all $i\in \{1,\ldots, m \}$ we will have $\alpha_{t,i}\rightarrow 0$ while $\gamma_{t,i}=\alpha_{i}$ will remain constant. In the end, if $\alpha_t$ is the representation of $P_t$, we can only claim that: \begin{equation}
	\alpha_{t,i}\leq \gamma_{t,i},
\end{equation} 
as mentioned in \eqref{feb13}.

The following lemma characterizes the behavior of the partial rescaling. It can be seen as a discrete analog to the classical results about Wulff shapes.

\begin{lem}\label{super easy lemma}
	Suppose $P\in\po_\mu$ and $I\subset \{1,\ldots, m\}$ is a nonempty and not full index set. Let $\alpha$ be its dual representation. Consider the $m-$tuple $\gamma_t$ defined by:
\begin{equation}
\begin{split}
	\gamma_{t,i}&=\alpha_i \text{ if } i\in I, \\ 
	\gamma_{t,i}&=t\alpha_i  \text{ if } i\notin I.
\end{split}
\end{equation}
Let $P_t$ be the dual Wulff shape of $\gamma_t$, where $t\in(0,1]$. Let $\alpha_t$ denote its representation. Then,
\begin{equation}\begin{split}
t\alpha_i\leq \alpha_{t,i}\leq \alpha_i \text{ for } i\in I,  \\
	\alpha_{t,i}=t\alpha_i  \text{ if } i\notin I.
	\end{split}
\end{equation}

\end{lem}
\begin{proof}
	Let $i\notin I$. By the definition of the representation,
	\begin{equation}
		\alpha_{t,i}=h_{P^*_{t}}(v_i).
	\end{equation}
	From \eqref{convex hull form of pt}, we obtain that $tP^*\subset P^*_{t}$. Thus,
	\begin{equation}\label{feb13.1}
		\alpha_{t,i}=h_{P^*_{t}}(v_i)\geq th_{tP^*}(v_i)=t\alpha_i.
	\end{equation}
	On the other hand, from the definition of the dual Wulff shape we have that \begin{equation}
		P^*_{t}\subset P^*\cap H^-(t\alpha_i,v_i),
	\end{equation} which implies that
	\begin{equation}\label{feb13.2}
		\alpha_{t,i}=h_{P^*_{t}}(v_i)\leq t\alpha_i.
	\end{equation}
	Combining, \eqref{feb13.1} and \eqref{feb13.2} we obtain that for $i\notin I$
	\begin{equation}
		\alpha_{t,i}=t\alpha_i.
	\end{equation}
	
	If $i\in I$, then, as $tP^*\subset P^*_{t}\subset P^*$, we obtain
	\begin{equation}
		t\alpha_i\leq \alpha_{t,i}\leq \alpha_i.
	\end{equation}
\end{proof}

\section{The Partial Rescaling of a Polytope and the Functional}\label{section proof}
We now turn our attention to the key lemma related to the partial rescaling. Under the assumption of the weak Aleksandrov relation, as we have noted, we can no longer claim that for any sequence maximizing the functional, there is a lower bound on the inner to outer radius ratio. The following lemma provides us with the tool to overcome this difficulty. It establishes that, with proper assumptions, the partial rescaling does not decrease the value of the functional.

\begin{lem}\label{alpha beh}
Suppose $P\in\po_\mu$ and $I\subset \{1,\ldots, m\}$ is a nonempty and not full index set. Let $\alpha$ be its dual representation. Consider the $m$-tuple $\gamma_t$ defined by
\begin{equation}
\begin{split}
	\gamma_{t,i}&=\alpha_i \text{ if } i\in I, \\ 
	\gamma_{t,i}&=t\alpha_i  \text{ if } i\notin I.
\end{split}
\end{equation}
Let $P_t$ be the dual Wulff shape of $\gamma_t$. Suppose for some $0< t_0<1$ the following holds: \begin{equation}\label{Feb 16 assumption}
  \mu(\bigcup_{i\notin I}v_i)\geq\lambda(\balpha_{P_{t_0}}(\bigcup_{i\notin I}v_i)).
\end{equation}
Then, $\Phi(P_{t_0})\geq\Phi(P)$. 
\end{lem}

\begin{proof}
	Before we start to compare $\Phi(P_{t_0})$ with $\Phi(P)$, let us first analyze the behavior of the radial functions under the partial rescaling. Let $t$ be any value between $t_0$ and $1$. First, we claim that for any $u\in\sn$: \begin{equation}\label{formula for rho}
  \rho_{P_t^*}(u)= \min\Bigg\{\left\{\frac{\alpha_i}{u\cdot v_i}\large \,\middle\vert\, i\in I, u\cdot v_i>0 \right\}\bigcup\left\{\frac{t\alpha_i}{u\cdot v_i}\large \,\middle\vert\, i\notin I, u\cdot v_i>0 \right\}\Bigg\}.
\end{equation}
The proof of this claim is the same as the proof of Lemma \ref{technical estimate 1}, where instead of \eqref{form in 4.2} one should use \eqref{convex hull form of pt}.
We also notice that the equation \eqref{convex hull form of pt} implies the following relation:
\begin{equation}\label{some relation 1}
	\balpha_{P_{t_2}}(\bigcup_{i\notin I}v_i) \subset \balpha_{P_{t_1}}(\bigcup_{i\notin I}v_i),
\end{equation}
for $0<t_1<t_2\leq 1$. In fact, from \eqref{convex hull form of pt} it is straightforward to verify that
\begin{equation}\label{some relation 2}
	\bigcup_{0<t\leq1} \balpha_{P_{t}}(\bigcup_{i\notin I}v_i)=\bigcup_{i\notin I}({v_i})_{\frac \pi 2}.
\end{equation}

Now, examining the equations \eqref{formula for rho}, \eqref{some relation 1}, and \eqref{some relation 2}, we can separate three possible behaviors of $\rho_{P_t^*}(u)$ as a function of $t$ for $t\leq 1$:

 \begin{itemize}
	\item If $u\in \balpha_{P}(\bigcup_{i\notin I}v_i)$, then $u\in \balpha_{P}(v_j)$ for some $j\notin I$. Thus, $r_{P}(v_j)\in H_{P}(u)$  which is equivalent to $r_{P^*}(u)\in H_{P^*}(v_j)$. Thus, from Lemma \ref{technical estimate 1} we obtain\begin{equation}
		\rho_{P^*}(u)=\frac {\alpha_j} {u\cdot v_j},
	\end{equation}
	where $j\notin I$, and, in particular, the minimum in \eqref{formula for rho} is attained at $j\notin I$. Therefore, from \eqref{formula for rho} we obtain that: \begin{equation}\label{4.24sim}
		\rho_{P_t^*}(u)=t\rho_{P^*}(u).
	\end{equation}
	\item If $u\notin \bigcup_{i\notin I}({v_i})_{\frac \pi 2}$, then $u\cdot v_i\leq 0$ for all $i\notin I$. Thus, from equation \eqref{formula for rho} (or equation \eqref{some relation 2}) we obtain that:
	\begin{equation}\label{4.25sim}
		\rho_{P_t^*}(u)=\rho_{P^*}(u).
	\end{equation}
	\item If $u\in \bigcup_{i\notin I}({v_i})_{\frac \pi 2} \setminus \balpha_{P}(\bigcup_{i\notin I}v_i)$, then by applying equations \eqref{formula for rho}, \eqref{some relation 1}, and \eqref{some relation 2} in a way similar to the previous two cases, we obtain that as $t$ decreases, $\rho_{P_t^*}(u)$ remains constant up until the first moment when $u\in \balpha_{P_t}(\bigcup_{i\notin I}v_i)$, after which it starts to scale. Let $t(u)$ denote the maximum $t$ for which $u\in \balpha_{P_{t}}(\bigcup_{i\notin I}v_i)$. We obtain:
	\begin{equation}\label{4.26sim}
	\begin{split}
		\rho_{P_t^*}(u)= \rho_{P^*}(u) \text{ when }  t\in[t(u),1], \\
		\rho_{P_t^*}(u)= \frac {t} {t(u)} \rho_{P^*}(u) \text{ when }  t\in(0,t(u)]. \\
	\end{split}
	\end{equation}
	\end{itemize}

By looking at the dual, using Lemma \ref{super easy lemma} and the convex hull equation \eqref{convex hull form of pt} for $P$, we obtain the following behavior for $\rho_{P_t}(v_i)$:
\begin{itemize}
	\item If $i\notin I$, then \begin{equation}\label{4.27sim}
		\rho_{P_t}(v_i)=\frac {\rho_{P}(v_i)} t.
	\end{equation}
	\item If $i\in  I$, then $\rho_{P_t}(v_i)$ is non-decreasing as $t$ decreases and \begin{equation}\label{4.28sim}
		\rho_{P_t}(v_i)\leq\frac {\rho_{P}(v_i)} t.
	\end{equation}
\end{itemize}
Now with the help of equations \eqref{4.24sim}-\eqref{4.28sim} regarding the behavior of $\rho_{P_t}$ and $\rho_{P^*_t}$, we would like to compute $\Phi(P_{t_1})-\Phi(P_{t_2})$ for some $0<t_1<t_2\leq 1$. From equations \eqref{some relation 1} and \eqref{some relation 2}, we can separate $\Phi(P_{t_1})-\Phi(P_{t_2})$ into the following terms:
	\begin{equation}\label{very huge formula 2}\begin{split}
  &\Phi(P_{t_1})-\Phi(P_{t_2})= \\
   &\int_{\sn}\log(\rho_{P_{t_1}})-\log(\rho_{P_{t_2}})d\mu+\\&\int_{\balpha_{P_{t_2}}(\bigcup_{i\notin I}v_i)}	\log(\rho_{{P^*_{t_1}}})-\log(\rho_{{P^*_{t_2}}})d\lambda +\\ &\int_{\sn\setminus{\bigcup_{i\notin I}({v_i})_{\frac \pi 2}}}\log(\rho_{{P^*_{t_1}}})-\log(\rho_{{P^*_{t_2}}})d\lambda+\\
   &\int_{\bigcup_{i\notin I}({v_i})_{\frac \pi 2}\setminus\balpha_{P_{t_1}}(\bigcup_{i\notin I}v_i)}\log(\rho_{{P^*_{t_1}}})-\log(\rho_{{P^*_{t_2}}})d\lambda+\\
   &\int_{\balpha_{P_{t_1}}(\bigcup_{i\notin I}v_i)\setminus\balpha_{P_{t_2}}(\bigcup_{i\notin I}v_i)}\log(\rho_{{P^*_{t_1}}})-\log(\rho_{{P^*_{t_2}}})d\lambda.
\end{split}
\end{equation}
Now, we are going to estimate each term in the above equation.
\vspace{2mm}

\textit{The first term}\\
First, we look at the integral with respect to $\mu$. From \eqref{4.27sim} and since $\rho_{P_t}(v_i)$ is non-decreasing for $i\in I$, see \eqref{4.28sim}, we obtain that: 
	\begin{equation}\label{mu part}
	\begin{split}
		\int_{\sn}\log(\rho_{P_{t_1}})-\log(\rho_{{P_{t_2}}})d\mu= \\ \sum_{i\in I} \big ( \log(\rho_{P_{t_1}}(v_i))-\log(\rho_{{P_{t_2}}}(v_i))\big)\mu(v_i)+\sum_{i\notin I} \big ( \log(\rho_{P_{t_1}}(v_i))-\log(\rho_{{P_{t_2}}}(v_i))\big)\mu(v_i) \geq \\  \sum_{i\notin I} \big ( \log(\rho_{P_{t_1}}(v_i))-\log(\rho_{{P_{t_2}}}(v_i))\big)\mu(v_i) \geq \\
		\log(\frac{t_2}{t_1})\mu(\bigcup_{i\notin I}v_i).
	\end{split}
\end{equation}
\vspace{2mm}

\textit{The second term}\\
There are two possibilities here. First, recalling \eqref{some relation 1}, we have that  \begin{equation}
	{\balpha_{P}(\bigcup_{i\notin I}v_i)}\subset {\balpha_{P_{t_2}}(\bigcup_{i\notin I}v_i)}\subset {\balpha_{P_{t_1}}(\bigcup_{i\notin I}v_i)} .
\end{equation} If we suppose that $u\in {\balpha_{P}(\bigcup_{i\notin I}v_i)}$ we immediately obtain from \eqref{4.24sim} that 

\begin{equation}\begin{split}
	\log(\rho_{{P^*_{t_2}}}(u))-\log(\rho_{{P^*_{t_1}}}(u))=\\
	\log(\frac {t_1}{t_2}). 
\end{split}
\end{equation}
Alternatively, if $u\in {\balpha_{P_{t_2}}(\bigcup_{i\notin I}v_i)}\setminus {\balpha_{P}(\bigcup_{i\notin I}v_i)}$, then $t_1\leq t(u)$ and $t_2\leq t(u)$ and we use \eqref{4.26sim} to obtain the same result: 
\begin{equation}\begin{split}
	\log(\rho_{{P^*_{t_2}}}(u))-\log(\rho_{{P^*_{t_1}}}(u))=\\
	\log(\frac {t_1} {t(u)}\rho_{P^*}(u))-\log(\frac {t_2} {t(u)}\rho_{P^*}(u))=\\
	\log(\frac {t_1}{t_2}). 
\end{split}
\end{equation}
Combining the two, for the second term, we obtain 
\begin{equation}
	\int_{\balpha_{P_{t_2}}(\bigcup_{i\notin I}v_i)}	\log(\rho_{{P^*_{t_1}}})-\log(\rho_{{P^*_{t_2}}})d\lambda 
	 =\log(\frac {t_1}{t_2})\lambda(\balpha_{P_{t_2}}(\bigcup_{i\notin I}v_i)). 
\end{equation}
\vspace{2mm}

\textit{The third term}\\ From \eqref{4.25sim}, the third term is computed as: 
\begin{equation}
\begin{split}
	\int_{\sn\setminus{\bigcup_{i\notin I}({v_i})_{\frac \pi 2}}}\log(\rho_{{P^*_{t_1}}})-\log(\rho_{{P^*_{t_2}}})d\lambda= 0.
\end{split}
\end{equation}
\vspace{2mm}

\textit{The fourth and the fifth terms}\\
Now, we want to estimate the last two terms. As in \eqref{4.26sim}, given \begin{equation}
	u\in {\bigcup_{i\notin I}({v_i})_{\frac \pi 2}\setminus\balpha_{P_{t_2}}(\bigcup_{i\notin I}v_i)},
\end{equation} recall that $t(u)$ denote the maximum $t$ such that $u\in \balpha_{P_t}(\bigcup_{i\notin I}v_i)$. Notice that if \begin{equation}
	u\in {\bigcup_{i\notin I}({v_i})_{\frac \pi 2}\setminus\balpha_{P_{t_1}}(\bigcup_{i\notin I}v_i)}, \end{equation} from \eqref{some relation 1} and \eqref{some relation 2} we obtain that $t(u)< t_1< t_2$. And, thus, since $t_1,t_2\in[t(u),1)$, from \eqref{4.26sim} we obtain: 
	\begin{equation}
		\log(\rho_{{P^*_{t_1}}}(u))-\log(\rho_{{P^*_{t_2}}}(u))= \\
		\log(\rho_{{P^*_{}}}(u))-\log(\rho_{{P^*_{}}}(u))=0,
	\end{equation}
	which implies that the fourth term is zero. On the other hand, if \begin{equation}
	u\in {\balpha_{P_{t_1}}(\bigcup_{i\notin I}v_i)\setminus\balpha_{P_{t_2}}(\bigcup_{i\notin I}v_i)},
\end{equation} then $t(u)\in [t_1,t_2)$. Thus, from \eqref{4.26sim} we obtain:
\begin{equation}\label{integral part equation}
\begin{split}
	\lvert\int_{\balpha_{P_{t_1}}(\bigcup_{i\notin I}v_i)\setminus\balpha_{P_{t_2}}(\bigcup_{i\notin I}v_i)}\log(\rho_{{P^*_{t_1}}})-\log(\rho_{{P^*_{t_2}}})d\lambda\rvert = \\
	\lvert\int_{\balpha_{P_{t_1}}(\bigcup_{i\notin I}v_i)\setminus\balpha_{P_{t_2}}(\bigcup_{i\notin I}v_i)}\log(\frac {t_1} {t(u)} \rho_{P^*}(u))-\log(\rho_{{P^*_{t}}}(u))d\lambda\rvert \leq \\
	\lambda\big(\balpha_{P_{t_1}}(\bigcup_{i\notin I}v_i)\setminus\balpha_{P_{t_2}}(\bigcup_{i\notin I}v_i)\big)\lvert\log(\frac {t_1}{t_2})\rvert.
\end{split}
\end{equation}
Overall, we established \eqref{integral part equation} using the assumption that $0<t_1<t_2\leq1.$ Let us now investigate \eqref{integral part equation} under the additional assumption that $t_1$ is approximately $t_2$. 

From the continuity of measure $\lambda$ and since $\lambda$ an is absolutely continuous measure, as $t_1\rightarrow t_2^-$ (with $t_1<t_2$), we find that  \begin{equation}
  \lambda\big(\balpha_{P_{t_1}}(\bigcup_{i\notin I}v_i)\setminus\balpha_{P_{t_2}}(\bigcup_{i\notin I}v_i)\big)\rightarrow \lambda\big (\partial (\balpha_{P_{t_2}}(\bigcup_{i\notin I}v_i))\big)=0
\end{equation}
where $\partial$ denotes the boundary of the set in $\sn$. So, in particular, given any $\eps>0$ for all $t_1$ sufficiently close enough to $t_2$, the right side at the end of \eqref{integral part equation} is less than  $\eps\lvert\log (\frac {t_1}{t_2})\rvert$.

Combining all of the  estimates for different terms of \eqref{very huge formula 2}, we obtain that given any $t_2\in (0,1]$ and $\eps>0,$ we can pick $\delta>0$ with $t_2-\delta>0$ such that for all $t_1\in [t_2-\delta,t_2]$:
\begin{equation}\label{most important equation}
\begin{split}
\Phi(P_{t_1})-\Phi(P_{t_2})&\geq \log(\frac{t_2}{t_1})\mu(\bigcup_{i\notin I}v_i)+\log(\frac {t_1}{t_2})\lambda(\balpha_{P_{t_2}}(\bigcup_{i\notin I}v_i))-\eps\lvert\log \frac {t_1}{t_2}\rvert \\ &= 
\log(\frac {t_2} {t_1})\Big (\mu(\bigcup_{i\notin I}v_i)-\lambda(\balpha_{P_{t_2}}(\bigcup_{i\notin I}v_i))-\eps\Big ).
\end{split}
\end{equation}

Now, we are going to analyze two special cases of previous equation guided by the assumption \eqref{Feb 16 assumption}. 
\vspace{2mm}

\textit{Case 1. Given $t_2\in (0,1]$, suppose $\mu(\bigcup_{i\notin I}v_i)>\lambda(\balpha_{P_{t_2}}(\bigcup_{i\notin I}v_i))$}.

\vspace{2mm}
Then \eqref{most important equation}, implies that we can pick $\delta>0$ such that for all $t_1\in[t_2-\delta,t_2)$, $\Phi(P_{t_1})>\Phi(P_{t_2})$. 

\vspace{2mm}
\textit{Case 2. Given $t_2\in (0,1]$, suppose $\mu(\bigcup_{i\notin I}v_i)=\lambda(\balpha_{P_{t_2}}(\bigcup_{i\notin I}v_i))$}.

\vspace{2mm}
Then, from \eqref{Feb 16 assumption}, \eqref{some relation 1} and \eqref{some relation 2}, for $t_1\in(0,t_2]:$ \begin{equation}
  \lambda\big(\balpha_{P_{t_1}}(\bigcup_{i\notin I}v_i)\setminus\balpha_{P_{t_2}}(\bigcup_{i\notin I}v_i)\big)=0.
\end{equation}
This forces the last part of \eqref{integral part equation} to be equal to zero. Which in turn implies that
the equation \eqref{most important equation}, under this particular assumptions, refines to the following:
\begin{equation}
\begin{split}
\Phi(P_{t_1})-\Phi(P_{t_2})&\geq \log(\frac{t_2}{t_1})\mu(\bigcup_{i\notin I}v_i)+\log(\frac {t_1}{t_2})\lambda(\balpha_{P_{t_2}}(\bigcup_{i\notin I}v_i)) \\ &= 
\log(\frac {t_2} {t_1})(\mu(\bigcup_{i\notin I}v_i)-\lambda(\balpha_{P_{t_2}}(\bigcup_{i\notin I}v_i)))=0.
\end{split}
\end{equation}
Therefore, in this particular case, for all $t_1\in (0,t_2]$ we obtain that $\Phi(P_{t_1})\geq\Phi(P_{t_2})$.

We are now fully prepared to conclude the proof. Recall that we were given $t_0<1$ such that \begin{equation}\label{the last step}
  \mu(\bigcup_{i\notin I}v_i)\geq\lambda(\balpha_{P_{t_0}}(\bigcup_{i\notin I}v_i)).
\end{equation}   
From \eqref{some relation 1}, we obtain that the same statement holds for any $t_2\in[t_0,1]$:
\begin{equation}
  \mu(\bigcup_{i\notin I}v_i)\geq\lambda(\balpha_{P_{t_2}}(\bigcup_{i\notin I}v_i)).
\end{equation} 
Suppose now that $\Phi(P_{t_0})<\Phi(P_1)$. Since $\Phi(P_t)$ is continuous, let $t_2\in(t_0,1]$ be the smallest value such that $\Phi(P_{t_2})=\Phi(P_1)$. Then, for all $t_1\in [t_0,t_2)$, $\Phi(P_{t_1})<\Phi(P_{t_2})$. If 
\begin{equation}
	\mu(\bigcup_{i\notin I}v_i)>\lambda(\balpha_{P_{t_2}}(\bigcup_{i\notin I}v_i))
\end{equation} 
then a contradiction arises from Case 1. If
\begin{equation}
	\mu(\bigcup_{i\notin I}v_i)=\lambda(\balpha_{P_{t_2}}(\bigcup_{i\notin I}v_i))
\end{equation}
then, from Case 2, we would obtain that $\Phi(P_{t_0})=\Phi(P_{t_2})$, which is also a contradiction. Therefore, $\Phi(P_{t_0})\geq \Phi(P_1)$, which was the  desired.
  \end{proof}
The next Lemma is a core component of the proof of Theorem \ref{main}, where we utilize two previous technical results: Lemma \ref{4.14} and Lemma \ref{alpha beh}. Before we begin with the proof, let us first try to explain its statement in a more intuitive way. Suppose we are given some $P\in\po_\mu$, and a nonempty and not full index set $I\subset \{1,\ldots, m\}$. Suppose for this index set we have the following: 
\begin{gather}
	0<L^*\leq U^*<L\leq U=1 \quad \\
	\frac L U \approx 1, \text{ } \frac {U^*} L \approx 0,\text{ } \frac {L^*} {U^*} \approx 1.
\end{gather}
Suppose $\mu$ is weakly Aleksandrov related to $\lambda$ and $\alpha$ is the uniform weak Aleksandrov constant for $\mu$ and $\lambda$ given by Proposition \ref{uniform constant}. Then, Lemma \ref{rescaling} claims that we can find a new polytope $P_{\frak r}\in\po_\mu$  with $\Phi(P_{{\frak r}})\geq \Phi(P)$ such that:
\begin{gather}
	0<L_{\frak r}^*\leq U_{\frak r}^*<L_{\frak r}\leq U_{\frak r}=1 \quad \\
	\frac {L_{\frak r}} {U_{\frak r}} \approx 1, \text{ } \frac {U_{\frak r}^*} {L_{\frak r}} \approx \cos(\frac \pi 2 -\alpha),\text{ } \frac {L_{\frak r}^*} {U_{\frak r}^*} \approx 1.
\end{gather}
Vaguely speaking, Lemma \ref{rescaling} allows us to construct a new polytope with coefficients that are closer together, without decreasing the functional value.

Even though the Aleksandrov condition is not stated explicitly, it is embedded in the following Lemma as part of the assumption about: \begin{equation}
	\mu(\bigcup_{i\notin I}v_i)\leq\lambda(\bigcup_{i\notin I}(v_i)_{\frac\pi2-\alpha}).
\end{equation}
We again recall that we use the notation established in \eqref{convex hull notation}-\eqref{notation UL 2}.  

\begin{lem}\label{rescaling}
	Suppose $P\in\po_\mu$ is such that $\max_i\alpha_i=1$, and $I\subset \{1,\ldots, m\}$ is a nonempty and not full index set. Suppose \begin{equation}\label{like Alex}
		\mu(\bigcup_{i\notin I}v_i))\leq\lambda(\bigcup_{i\notin I}(v_i)_{\frac\pi2-\alpha})
	\end{equation} for some $  0 < \alpha < \frac \pi 2$. Suppose ${U^*} <L{\cos({\frac{\pi}{2}-\alpha})}$. In particular, $0<L^*\leq U^*<L\leq U=1$.  Then, there exists $P_{{\frak r}}\in\po_\mu$, such that:
	\begin{equation}
	\begin{split}
		U_{\frak r}&=1, \\
		L_{\frak r}&=L, \\
		U^*_{\frak r}&\leq L\cos(\frac \pi 2 - \alpha), \\
		L^*_{\frak r}&= \frac {L^*} {U^*} L\cos(\frac \pi 2 - \alpha).
	\end{split}
	\end{equation}
 Moreover, $\alpha_i=\alpha_{{\frak r},i}$ for $i\in I$ and $\Phi(P_{\frak r})\geq\Phi(P)$. 
\end{lem}

\begin{proof}
	We want to rescale $P$ for index set $I^c=\{1,\ldots,m\}\setminus I$ as:
	\begin{equation}\label{Feb 18 partially rescalling}
  P^*_t=\bigcap_{i\in I^c}H^-(\alpha_i,v_i)\bigcap_{i\notin I^c}H^-(t\alpha_i,v_i).
\end{equation}
Note that we partially rescale the polytope $P$ with respect to index set $I^c$, which is the opposite of the rescaling used in Lemma \ref{alpha beh} and Lemma \ref{super easy lemma}. (See \eqref{convex hull form of pt} as well.) Let
\begin{equation}\label{18 defenition of t}
	t_0=\frac {U^*} {L\cos(\frac \pi 2 - \alpha)}.
\end{equation}
From our assumptions, $0<t_0<1$. Our goal is to analyze $P^*_{t_0}$ and to confirm that $\Phi(P_{t_0})\geq \Phi(P)$ with the help of Lemma \ref{alpha beh}. Recall, the equation \eqref{formula for rho} from the proof of Lemma \ref{alpha beh}. Using it, we can write for $u\in\sn$: \begin{equation}\label{representation formula for proof}
  \rho_{P_t^*}(u)=\min\Bigg\{\left\{\frac{\alpha_i}{u\cdot v_i}\large \,\middle\vert\, i\in I^c, u\cdot v_i>0 \right\}\bigcup\left\{\frac{t\alpha_i}{u\cdot v_i}\large \,\middle\vert\, i\notin I^c, u\cdot v_i>0 \right\}\Bigg\}.
\end{equation}
	Moreover, from Lemma \ref{super easy lemma} we obtain that
	\begin{equation}\label{Feb 19 equation}\begin{split}
t\alpha_i\leq \alpha_{t,i}\leq \alpha_i \text{ for } i\in I^c,  \\
	\alpha_{t,i}=t\alpha_i  \text{ if } i\notin I^c.
	\end{split}
\end{equation} 
Thus, recalling the definitions of $L_t^*(I),U_t^*(I),L_t(I),U_t(I)$, it follows from previous equation that
 \begin{equation}\label{Feb 18, large equation}
  \begin{split}
  	&U_t=tU, \\
  	&L_t=tL, \\
  	U^*\geq &U^*_t\geq tU^*.
  \end{split}
\end{equation}
Notice that from the previous equation and the equation \eqref{18 defenition of t}: \begin{equation}\label{crucial estimate super}
	U^*_{t_0}\leq U^*=t_0L\cos(\frac \pi 2 - \alpha)= L_{t_0}\cos(\frac \pi 2 - \alpha).
\end{equation}
In particular, 
\begin{equation}
	U^*_{t_0}<L_{t_0}.
\end{equation}
Therefore, we can apply Lemma \ref{4.14} to $P_{t_0}$ to obtain:
\begin{equation}
 \lambda(\bigcup_{i\notin I}({v_i})_{\arccos\frac {U_{t_0}^*} {L_{t_0}}})\leq \lambda(\balpha_{P_{t_0}}(\bigcup_{i\notin I}v_i)).
\end{equation}
This, combined with the assumption \eqref{like Alex} and the estimate \eqref{crucial estimate super}, gives us the following:
\begin{equation}\label{4.21}
\mu(\bigcup_{i\notin I}v_i))\leq\lambda(\bigcup_{i\notin I}(v_i)_{\frac\pi2-\alpha})\leq \lambda(\bigcup_{i\notin I}({v_i})_{\arccos\frac {U_{t_0}^*} {L_{t_0}}})\leq \lambda(\balpha_{P_{t_0}}(\bigcup_{i\notin I}v_i)).
\end{equation}
Therefore, since $\balpha_{P_{t_0}}(\bigcup_{i\notin I}v_i)\bigcap\balpha_{P_{t_0}}(\bigcup_{i\in I}v_i)$ is a $\lambda$ measure zero set, as it has Lebesgue measure zero, and since $\mu$ and $\lambda$ have equal weights, we obtain from the equation above that the reverse holds for $\balpha_{P_{t_0}}(\bigcup_{i\in I}v_i)$, that is 
\begin{equation}
\begin{split}
\mu(\bigcup_{i\in I}v_i)\geq\lambda(\balpha_{P_{t_0}}(\bigcup_{i\in I}v_i)).
\end{split}
\end{equation}
By rewriting this to
\begin{equation}\label{4.22}
\begin{split}
\mu(\bigcup_{i\notin I^c}v_i)\geq\lambda(\balpha_{P_{t_0}}(\bigcup_{i\notin I^c}v_i)),
\end{split}
\end{equation}
we can apply Lemma \ref{alpha beh} (again, notice that we partially rescale index set $I^c$) to conclude:
\begin{equation}\label{Feb 18.1}
	\Phi(P_{t_0})\geq\Phi(P).
\end{equation}

We claim that $P_{\frak r}:=t_0P_{t_0}$ is the desired polytope. To prevent possible confusion, let us remark that ${\frak r}$ in $P_{\frak r}$ is just a convenient notation and does not correspond to partial rescaling of $P$ as in  \eqref{Feb 18 partially rescalling} by a factor $\frak r$. On the other hand, $P_{t_0}$ is a partial rescaling of $P$ given by formula  \eqref{Feb 18 partially rescalling}.

 From \eqref{Feb 18.1} we  obtain,
\begin{equation}
	\Phi(P_{\frak r})=\Phi(P_{t_0})\geq \Phi(P).
\end{equation}
What remains is to establish the values of $U^*_{\frak r}, L^*_{\frak r}, L_{\frak r}$ as well as $\alpha_{{\frak r},i}$ for $i\in I$.

Firstly, we notice from \eqref{some geometric obvious relation}: \begin{equation}\label{4.5777}
  L_t^*\geq \min_{u\in\sn}{h_{P_t^*}(u)}=\min_{u\in\sn}{\rho_{P_t^*}(u)}.
\end{equation}
We are interested in whether $\rho_{P_t^*}(u)$ decreases. Notice that for $t\geq t_0$, we have \begin{equation}\label{so many equations}
  t\geq t_0=\frac {U^*} {L\cos(\frac \pi 2 - \alpha)}\geq\frac {L^*} {L\cos(\frac \pi 2 - \alpha)}>\frac{L^*}{L}. 
\end{equation}
Now, from the previous equation, for any $i\in I$ if $u\cdot v_i>0$, since $u\cdot v_i\leq 1$, we have \begin{equation}
  \frac{{t}\alpha_i}{u\cdot v_i}>\frac{L^*\alpha_i}{L}>L^*.
\end{equation}
We also have that for $i\notin I$ and $u\cdot v_i>0$: \begin{equation}
  \frac{\alpha_i}{u\cdot v_i}\geq L^*.
\end{equation}
Combining both previous equations and using \eqref{representation formula for proof}, we obtain that $\rho_{P^*_{t}}(u)\geq L^*$ for any $u\in\sn$ and for any $t\geq t_0$. Therefore, we can apply \eqref{4.5777} to deduce that $L_t^*\geq L^*$. This, combined with the fact that $L^*_{t}$ can only decrease as $t$ decreases and that $L_1^*=L^*$, imply that $L^*_t=L^*$ for $t\geq t_0$.

Summarizing this with \eqref{Feb 18, large equation}, we obtain that for $t\geq t_0$:  
\begin{equation}\label{everything combined}
  \begin{split}
  	&U_t=tU, \\
  	&L_t=tL,\\
  	U^*\geq &U^*_t\geq tU^*, \\
  	&L^*_t=L^*.
  \end{split}
\end{equation}
Recall that $\alpha_t$ is a representation for $P_t$. Notice that \begin{equation}
	\max_i\alpha_{t,i}= \max(U_t,U^*_t).
\end{equation} Now, since $U=1$, from the definition of $t_0$ and \eqref{everything combined} we obtain \begin{equation}
  U^*_{t}\leq U^* = t_0L\cos(\frac{\pi} 2 - \alpha)<t_0=t_0U=U_{t_0}.
\end{equation}
Therefore, combining two previous equations, we obtain
\begin{equation}\label{Feb 18 346}
	\max_i\alpha_{t,i}= \max(U_t,U^*_t)=U_{t}=t.
\end{equation} 
Recall that we defined $P_{\frak r}$ to be equal to $t_0P_{t_0}$. Thus, $P_{\frak r}^*=\frac {P_{t_o}^*}{t_0}$, and, therefore, $\alpha_{{\frak r},i}=\frac{\alpha_{{t_o},i}}{t_0}$. In particular, from \eqref{Feb 18 346}, we obtain that 
\begin{equation}
	\max_i\alpha_{\frak r,i}= \frac {\max_i\alpha_{t_0,i}}{t_0}=1.
\end{equation} 
Thus, from \eqref{everything combined}, we obtain for $P_{\frak r}=t_0P_{t_0}$: \begin{equation}
\begin{split}
&U_{\frak r}=\frac {U_{t_0}}{t_0}=1, \\
&L_{\frak r}= \frac {L_{t_0}} {t_0}= L,\\ 
&U^*_{\frak r}= \frac {U_{t_0}^*} {t_0} \leq \frac {U^*} {t_0} = L\cos(\frac \pi 2 - \alpha),\\
  &L^*_{\frak r}=\frac {L^*_{t_0}} {t_0}=\frac {L^*} {U^*} L\cos(\frac \pi 2 - \alpha).
 \end{split} 
\end{equation}

Now, it only remains to show that for $i\in I$, we have that $\alpha_i=\alpha_{\frak r,i}$. This immediately follows from \eqref{Feb 19 equation}:
\begin{equation}
	\alpha_{r,i}=\frac {\alpha_{t_0,i}} {t_0}=\alpha_{i}.
\end{equation}

\end{proof}

\section{Proof of the Main Result}

We are ready to start the proof Theorem \ref{main}. Our strategy is first to pick a sequence of polytopes that maximize the functional $\Phi$. Then, we will use Lemma \ref{rescaling} to modify this sequence, ensuring that it converges to a non-degenerate convex polytope. The proof heavily relies on the notations from \eqref{convex hull notation}-\eqref{notation UL 2} for varying index sets. We recall that, given a polytope with index $\fn$, as in $P_{\fn}$, and an index set $I\in\{1\dots m\}$ we write
\begin{equation}\begin{split}
  U_{\fn}(I):=&\max_{i\in I}\alpha_i, \\ L_{\fn}(I):=&\min_{i\in I}\alpha_i,\\ U_{\fn}^*(I):=&\max_{i\notin I}\alpha_i,\\ L_{\fn}^*(I):=&\min_{i\notin I}\alpha_i.
  \end{split}
\end{equation}
where $\alpha_\fn=(\alpha_{\fn,1},\dots,\alpha_{\fn,m})$ is the representation of $P_\fn$.

\begin{prop}\label{maximizer}
Suppose $\mu$ is a discrete measure not concentrated on a closed hemisphere and $\lambda$ is an absolutely continuous Borel measure.  Suppose $\mu$ is weakly Aleksandrov related to $\lambda$. Then there exists a sequence of polytopes $P_\fn\in\po_\mu$ maximizing $\Phi(\cdot)$, such that it converges to some $P\in\po_\mu$.	
\end{prop}
\begin{proof}
	Let $(P_\fn)_{\fn=1}^{\infty}$ be any sequence that maximizes the functional. For each $\fn$, let $\alpha_{\fn}$ be the representation of $P_\fn$. Rescale each $P_\fn$ so that $\max_{i}\alpha_{\fn,i}=1$. Since the set $\{v_i\mid 1\leq i \leq m\}$ is not contained in a closed hemisphere, it follows from  $\max_{i}\alpha_{\fn,i}=1$ that there exists $R>0$ such that for all $\fn$, $\mathfrak R_{P^*_{\fn}}<R$. Thus, we obtained a sequence that maximizes the functional and has a uniform bound on the outer radii of the duals.
	
	For every permutation $\sigma$ in $S_m$, where $S_m$ represents the set of all possible permutations of $m$ elements, we define the set $A_\sigma\subset  \mathbb N$ to contain all indices $\fn$ such that:
	\begin{equation}
		1= \alpha_{\fn,\sigma(1)} \geq \alpha_{\fn,\sigma(2)} \geq \ldots \alpha_{\fn,\sigma(m)} >0. 
	\end{equation}
	Since $\mathbb N=\cup_{\sigma\in S_m}A_\sigma$, there exists $\sigma\in S_m$ such that one of these sets is infinite. Without loss of generality, we can assume that the set $A_\sigma$ is infinite, corresponding to the identity permutation $\sigma$. We then take the subsequence of $(P_\fn)_{\fn=1}^{\infty}$ containing only elements in $A_\sigma$. Since we will never use the original sequence, we redefine the constructed subsequence to be $(P_\fn)_{\fn=1}^{\infty}$.
	
	Thus, we obtain a sequence of polytopes $(P_\fn)_{\fn=1}^{\infty}$ maximizing the functional such that for each $\fn$:
		\begin{equation}\label{ordering of coefficients}
		1= \alpha_{\fn,1} \geq \alpha_{\fn,2} \geq \ldots \geq \alpha_{\fn,m} >0 .
	\end{equation} 
	Using Bolzano--Weierstrass theorem, we can pass to the subsequence, which we again redefine to be $(P_\fn)_{\fn=1}^{\infty}$, such that $P^*_\fn$ converges to some convex set $K$, which can be written as:
\begin{equation}\label{form of P main}
K=\bigcap_{i=1}^mH^-(\alpha_i,v_i),
\end{equation}
where $\alpha_i$ are given by
\begin{equation}
\lim_{\fn\rightarrow \infty}\alpha_{\fn,i}=\alpha_i.
\end{equation} 
Moreover, through repeated application of Bolzano--Weierstrass theorem, we can assume as well that for each $i\in\{1,\ldots, m-1\}$, there exists $\beta_i\in [0,1]$:
\begin{equation}\label{beta equation 1}
	\lim_{\fn\rightarrow \infty} \frac {\alpha_{\fn,{i+1}}} {\alpha_{\fn,i}} = \beta_i,
\end{equation}
since $0<\alpha_{\fn,{i+1}}\leq {\alpha_{\fn,i}}$	from \eqref{ordering of coefficients}. Notice that from \eqref{beta equation 1} we also trivially obtain the following equation:
\begin{equation}\label{beta equation 2}
	\lim_{\fn\rightarrow \infty} \frac {\alpha_{\fn,{i+k}}} {\alpha_{\fn,i}} = \prod_{0\leq j\leq k-1} \beta_{i+j},
\end{equation} Finally, the following holds for $\alpha
_i$:
\begin{equation}\label{inequalities for alpha}
	1= \alpha_{1} \geq \alpha_{2} \geq \ldots \geq \alpha_{m} \geq0. 
\end{equation} 

Note that if all coefficients $\alpha_i>0$, then $K$ contains the origin in its interior, and, thus, $P_{\fn}$ will converge to $K^*=P\in \po_\mu$, which is the desired. Suppose it is not the case. Then, in particular, from \eqref{inequalities for alpha}, we obtain that $\alpha_m=0$. Since $\alpha_1=1$, we obtain that at least one of the variables $\beta_i$ is equal to zero. 
 
 Now, we will construct index sets $I_j$ that correspond to what we call as \textit{different rates of convergence of} $\alpha_{\fn,i}$ \textit{to zero}. Suppose there are exactly $k$ indices, $i_0 \leq i_1 \leq \ldots \leq i_{k-1}$,  for which $\beta_{i_j}$ is equal to zero, where $j\in\{0,\ldots,k-1\}$. For convenience, let $i_k=m$. We define:
 \begin{equation}
 \begin{split}
 	I_0=&\{1,\ldots, i_0\},\\
 	I_1=&\{i_0+1,\ldots ,i_1\},\\
 	\vdots \\
 	I_k=&\{i_{k-1}+1,\ldots ,m\}.
 \end{split}
 \end{equation}  
Notice that by construction, these sets are nonempty and not full, and their union is $\{1,\ldots, m\}$. Moreover, from \eqref{ordering of coefficients}, \eqref{beta equation 1}, and from the definitions of sets $I_j$ and indices $i_j$, we obtain the following inequalities:
\begin{equation}\label{a lot of inequalities 1}
		1=U_\fn(I_0)\geq L_\fn(I_0) \geq U_\fn(I_1)\geq L_{\fn}(I_1)\geq\ldots \geq U_\fn(I_k)\geq L_\fn(I_k)>0,
	\end{equation}
\begin{equation}\label{a lot of inequalities 2}
	\lim_{\fn\rightarrow\infty}\frac{U_\fn(I_{j+1})}{L_{\fn}(I_{j})}=\lim_{\fn\rightarrow\infty}\frac{\alpha_{\fn,{i_j+1}}}{\alpha_{\fn,i_j}}=\beta_{i_j}=0,
\end{equation}	
\begin{equation}\label{a lot of inequalities 3}
	\lim_{\fn\rightarrow\infty}\frac{L_\fn(I_{j})}{U_{\fn}(I_{j})}=\lim_{\fn\rightarrow\infty}\frac{\alpha_{\fn,{i_j}}}{\alpha_{\fn,{i_{j-1}}+1}}> c_j \text{ for some constant } c_j>0.\end{equation}	

Intuitively, each $I_j$ contains indices of elements that converge to 0 at the same rate. For example, from \eqref{form of P main}, \eqref{a lot of inequalities 1}-\eqref{a lot of inequalities 3}, we see that $i\in I_0$ if and only if $\alpha_i>0$. On the other hand, elements with an index in the set $I_k$ converge to zero the fastest. We call such a sequence of polytopes to be \textit{degenerate of order k}. If we have a convergent sequence with limit $K$ that is degenerate of order 0, then for all $i\in \{1,\ldots,m\}$, $\alpha_i>0$ and, hence, $K^*\in \po_\mu$. Therefore, to prove the proposition, it is sufficient to show that for a given sequence that is degenerate of order $k>0$, we can always find a new sequence of polytopes that is degenerate of a strictly lesser order than $k$ and such that the new sequence still maximizes the functional. This will be our goal for the rest of the proof.

Let $I=I_0\cup I_1 \ldots \cup I_{k-1}$. Note that the set $\{v_i \mid i\notin I\}$ is contained in a closed hemisphere as otherwise $P_\fn^*$ would converge to zero everywhere, which would contradict our assumption that $\max_{i}\alpha_{\fn,i}=1$. Since $\mu$ and $\lambda$ are weak Aleksandrov related, using Proposition \ref{uniform constant}, we obtain that there exists a uniform weak Aleksandrov constant $\alpha>0$. Therefore, since the set $\{v_i \mid i\notin I\}$ is a closed set and contained in a closed hemisphere, we can write:
\begin{equation}\label{to use 1}
   \mu(\bigcup_{i\notin I}v_i))\leq\lambda(\bigcup_{i\notin I}(v_i)_{\frac\pi2-\alpha}).
\end{equation}
From \eqref{a lot of inequalities 2} we can find $N$ such that  $\forall\fn>N$,
\begin{equation}\label{to use 2}
	\frac{U_{\fn}^*(I)}{L_\fn(I)}=\frac{U_{\fn}(I_k)}{L_\fn(I_{k-1})} < \cos(\frac\pi 2 - \alpha).
\end{equation} 
It should also be noted that for all $\fn$ we have that $U_\fn(I) = U_\fn(I_0)=1$. Thus, after combining this with \eqref{to use 1} and \eqref{to use 2}, for each $\fn>N$, we can apply Lemma \ref{rescaling} to $P_\fn$ and the index set $I$ to construct the partially rescaled polytopes $P_{r,\fn}$. Let $(P_{r,\fn})_{\fn=N}^{\infty}$ be the newly obtained sequence of partially rescaled polytopes. Let $\alpha_{r,\fn}$ be the representations of $P_{r,\fn}$.

First of all, note that, according to Lemma \ref{rescaling}, we have $\Phi(P_{r,\fn})\geq \Phi(P_{\fn})$. Therefore,  the newly obtained sequence sequence still maximizes the functional $\Phi(\cdot)$. Secondly, from Lemma \ref{rescaling}, we also have that  for $i\in I$,
\begin{equation}\label{sdjk}
	\alpha_{r,\fn,i}=\alpha_{\fn,i}.
\end{equation} We also note the following identities based on Lemma \ref{rescaling}:
\begin{equation}\label{ajohfh}
\begin{split}
	U_{r,\fn}(I)=&1, \\
	 L_{r,\fn}(I)=&L_\fn(I), \\
	U^*_{r,\fn}(I)\leq & L_{\fn}(I)\cos(\frac \pi 2 - \alpha), \\
	L^*_{r,\fn}(I)=& \frac {{L^*_\fn}(I)} {{U^*_\fn}(I) }L_\fn(I)\cos(\frac \pi 2 - \alpha).
\end{split}	
\end{equation}

From \eqref{ordering of coefficients}, \eqref{sdjk}, and \eqref{ajohfh}, we obtain: 
\begin{equation}\label{ordering of coefficients new}
		1 =\alpha_{\fn,1} = \alpha_{r,\fn,1} \geq \alpha_{\fn,2} =\alpha_{r,\fn,2} \geq \ldots = \alpha_{r,\fn,i_{k-1}} \geq \frac {U^*_{r,\fn}(I)} {\cos(\frac \pi 2 - \alpha)} > {U^*_{r,\fn}(I)}>0.
	\end{equation} 
	Therefore, sets $I_0, I_1, \ldots, I_{k-2}$, which contained coefficients with a rate of convergence up to $k-2$, remain the same for the newly constructed sequence of polytopes $(P_{r,\fn})_{\fn=N}^{\infty}$. Moreover, $I_{k-1}$ still contains the coefficients that converge to zero with a rate $k-1$. We will show that for the subsequence of the newly constructed sequence of rescaled polytopes, the coefficients from the index set $I_k$ converge with a rate $k-1$ as well.   
	
From \eqref{ajohfh} and \eqref{a lot of inequalities 3}, for coefficients in $I_k$ we have the following bound:
	\begin{equation}\label{everything converges}
		1\geq \frac{L^*_{r,\fn}(I)} {U^*_{r,\fn}(I)} \geq \frac {{L^*_\fn}(I)} {{U^*_\fn}(I) } = \frac {{L_\fn}(I_k)} {{U_\fn}(I_k) } > c_k>0.
	\end{equation}
	Moreover, from \eqref{ajohfh} and \eqref{a lot of inequalities 3}, we have:
	\begin{equation}\label{when will this end}
		\frac {L^*_{r,\fn}(I)} {\alpha_{r,\fn,i_{k-1}}} =  \frac {L^*_{r,\fn}(I)} {L_\fn(I)}= \frac {{L^*_\fn}(I)} {{U^*_\fn}(I) }\cos(\frac \pi 2 - \alpha)>c_k \cos(\frac \pi 2 - \alpha) >0.
	\end{equation}

	While it is true that for $i\in I$, coefficients $\alpha_{r,\fn,i}$ converge as they are equal to $\alpha_{\fn,i}$, they might not do so for $i\notin I$. By applying Bolzano--Weierstrass theorem and choosing new subsequence, we can ensure that they converge. Moreover, \eqref{everything converges} guarantees that we can apply the same construction as in the beginning to ensure that all the ratios between the elements converge for $i\notin I$ to some values in the range $[c_k,\frac 1 {c_k}]$. Since the proof does not change, for the sake of notational simplicity, we assume that all elements, as well as all their ratios, converge for $i\notin I$ in $(P_{r,\fn})_{\fn=N}^{\infty}$. As at the beginning of the proof, we pick some subsequence, so that for some permutation $\sigma$ of elements in $I_k$, we have the following based on \eqref{ordering of coefficients new}:
	\begin{equation}\label{ordering of coefficients new 2}
		1= \alpha_{r,\fn,1} \geq \alpha_{r,\fn,2} \geq \ldots \geq \alpha_{r,\fn,i_{k-1}}  > {U^*_{r,\fn}(I)}=\alpha_{r,\fn,\sigma(i_{k-1}+1)}\geq \ldots \geq \alpha_{r,\fn,\sigma(m)} >0.
	\end{equation} 
	
This, combined with \eqref{when will this end}, establishes that all coefficients with index $i\in I_k$ converge the same as coefficient $\alpha_{r,\fn,i_{k-1}}$:
\begin{equation}
	\lim_{\fn\rightarrow \infty} \frac {\alpha_{r,\fn,i}}{\alpha_{r,\fn,i_{k-1}}} \geq  \lim_{\fn\rightarrow \infty}  \frac {L^*_{r,\fn}(I)} {\alpha_{r,\fn,i_{k-1}}} >c_k \cos(\frac \pi 2 - \alpha) >0.
\end{equation}
Since $\alpha_{r,\fn,i_{k-1}}$ converges with a rate $k-1$, we have constructed a sequence of polytopes that is degenerate of order $k-1$. This finishes the proof.
\end{proof}

The proof of Theorem \ref{main} immediately follows by an application of Theorem 8.2 in \cite{GIP}.

\begin{proof}[\textbf{Theorem \ref{main}} Proof.] By Proposition \ref{maximizer}, there exists $P\in\po_\mu\subset\kno$ that maximizes the functional. Since $\mu$ is a Borel measure and $\lambda$ is an absolutely continuous Borel measure, we infer from Theorem 8.2 in \cite{GIP} that $\mu=\lambda(P,\cdot)$. 
\end{proof}

To conclude, let us prove another Proposition which characterizes the bound on the inner to outer radius ratio for the solution to the Gauss Image Problem. The existence of the uniform constant comes from Proposition \ref{uniform constant}.

\begin{prop}\label{weak Aleks bound}
Suppose $\mu$ is a discrete measure that is not concentrated on a closed hemisphere and $\lambda$ is an absolutely continuous Borel measure.  Suppose $\mu$ is weak Aleksandrov related to $\lambda$.  Let $\alpha$ be their uniform weak Aleksandrov constant. Then, there exists a polytope solution $P\in\po_\mu$ to the Gauss Image Problem such that the ratio $\frac{\mathfrak r_P}{\mathfrak R_p}$ is bounded from below by a constant depending only on vectors $v_i$ and the uniform weak Aleksandrov constant $\alpha$. Apart from $\alpha$ and vectors $v_i$, this constant is independent of $\lambda$.
 \end{prop}
\begin{proof}
By Theorem \ref{main}, there exists $P\in\po_\mu$ solving the Gauss Image Problem for measures $\mu$ and $\lambda$. Consider any sequence of solutions $P_\fn\in\po_\mu$, with $\max_i\alpha_{\fn,i}=1$, that maximizes the ratio\begin{equation}
  \frac{\min_i\alpha_{\fn,i}}{\max_i\alpha_{\fn,i}}.
\end{equation}
By compactness, there exists a subsequence converging to the body $P\in\po_\mu$. Since $P$ still maximizes the functional, it is a solution by Theorem 8.2 in \cite{GIP}. Let $\alpha$ be the representation for $P$. We also have that $\max_i\alpha_{i}=1$ and if $\beta$ is the representation for any other solution $P'$ to the Gauss Image Problem, then 
\begin{equation}\label{ration to improve}
	\frac{\min_i\alpha_{i}}{\max_i\alpha_{i}} \geq \frac{\min_i\beta_{i}}{\max_i\beta_{i}}.
\end{equation}

For this $P$, reorder the index set so that $1=\alpha_1\geq\alpha_{2}\geq...\geq\alpha_m>0$. Define $I_l=\{1\dots l\}$. Let $k>1$ be the integer such that the vectors $\{v_i\mid i\notin {I_{k}}\}$ are contained in a closed hemisphere, but the vectors $\{v_i\mid i\notin {I_{k-1}}\}$ are not. Clearly, $k<m-1$. Then, for any index set $I_l$ with $l\geq k$, we have that $\{v_i\mid i\notin {I_{l}}\}$ are contained in a closed hemisphere. Thus, from Proposition \ref{uniform constant}, we have \begin{equation}
  \mu(\bigcup_{i\notin I_l}v_i))\leq\lambda(\bigcup_{i\notin I_l}(v_i)_{\frac\pi2-\alpha}),
\end{equation}
where $\alpha$ is the uniform weak Aleksandrov constant. 

Suppose \begin{equation}\label{wrong equation}
	{U^{*}(I_l)} <{L(I_l)}{\cos({\frac{\pi}{2}-\alpha})}.
\end{equation} Then, we are able to apply Lemma \ref{rescaling} to find another body $P_r$ such that 
\begin{equation}\label{4.96}
	\begin{split}
		L_r(I_l)&=L(I_l),\\
		L^{*}_r(I_l)&= \frac {L^{*}(I_l)} {U^{*}(I_l)} L(I_l)\cos(\frac \pi 2 - \alpha)>L^{*}(I_l). 
	\end{split}
\end{equation}
Notice that $L^*(I_l)<L(I_l)$. This, combined with \eqref{4.96}, gives
\begin{equation}
	\min \big (L_r(I_l)),L^{*}_r(I_l)\big)>L^{*}(I_l).
\end{equation}
So, in particular, if $\alpha_r$ is the representation for $P_r$, then 
\begin{equation}
	\min_i \alpha_{r,i}= \min(L^{*}_r(I_l),L_r(I_l))>L^{*}(I_l)=\min_i\alpha_i.
\end{equation} Therefore, we obtain that\begin{equation}\label{compare}
   \frac{\min_i\alpha_{r,i}}{\max_i\alpha_{r,i}}> \frac{\min_i\alpha_{i}}{\max_i\alpha_{i}}.
\end{equation}
Since $\Phi(P_r)\geq\Phi(P)$, we have that $P_r$ is still a solution to the Gauss Image Problem by Theorem 8.2 from \cite{GIP}. Therefore, \eqref{compare} contradicts \eqref{ration to improve}.

 Thus, we obtain that for all $l\geq k$ the opposite to \eqref{wrong equation} holds, that is \begin{equation}
 {U^{*}(I_l)} \geq{L(I_l)}{\cos({\frac{\pi}{2}-\alpha})}.
\end{equation}
In particular, since $1=\alpha_1\geq\alpha_{2}\geq...\geq\alpha_m>0$ we obtain for $l\geq k$
\begin{equation}
	\alpha_{l+1}\geq \alpha_l \cos(\frac \pi 2 - \alpha).
\end{equation}
Thus, we obtain
\begin{equation}\label{finally to use in last lemma}
	\alpha_{m}\geq \alpha_k  \cos(\frac \pi 2 - \alpha)^{m-k}.
\end{equation}
And, therefore, \begin{equation}\label{adlfjnia}
	\min_{u\in\sn}\rho_{P^*}(u)=\alpha_m\geq \alpha_k\cos(\frac \pi 2 - \alpha)^{m-k}.
\end{equation}

 Now, consider $\{v_i\mid i\notin {I_{k-1}}\}$. We define $\gamma$ as following: \begin{equation}\begin{split}
 	\gamma=\inf\left\{{u \cdot v_i}\large \,\middle\vert\, u\in\sn, i\notin {I_{k-1}}, u \cdot v_i>0 \right\}.
 \end{split}
\end{equation}
Since $\{v_i\mid i\notin {I_{k-1}}\}$ are not contained in a closed hemisphere, we obtain that $\gamma>0$. Thus, from Lemma \ref{technical estimate 1} we have that for all $u\in \sn$
\begin{equation}\label{41}
\begin{split}
  \rho_{P^*}(u)=& \\ \min\left\{\frac{\alpha_i}{u\cdot v_i}\large \,\middle\vert\, i\in\{1,\ldots,m\}, u\cdot v_i>0 \right\}  \leq&\\
  \min\left\{\frac{\alpha_i}{u\cdot v_i}\large \,\middle\vert\, i\notin I_{k-1}, u\cdot v_i>0 \right\} \leq& \\
  \frac {\alpha_k}{\gamma}.
\end{split}
\end{equation}
Combining this with equation \eqref{adlfjnia} we obtain: \begin{equation}
  \frac {\mathfrak r_P}{\mathfrak R_P}=\frac {\mathfrak r_{P^*}}{\mathfrak R_{P^*}}= \frac{\min \rho_{P^*}(u)}{\max \rho_{P^*}(u)}\geq \frac {\alpha_k\cos(\frac \pi 2 - \alpha)^{m-k}}{\frac{\alpha_k}{\gamma}}\geq \gamma \cos(\frac \pi 2 - \alpha)^{m-k}.
\end{equation}
\end{proof}

It would be interesting to consider whether the above approach can be used to solve the Gauss Image Problem with the weak Aleksandrov condition when $\lambda$ is an absolutely continuous measure, and no additional discrete conditions are imposed on the measure $\mu$, as was done in \cite{GIP} for the classical Aleksandrov condition. The natural approach would be to discretize $\mu$ and to try to invoke Proposition \ref{weak Aleks bound}. Yet, we notice that the bound on the inner to outer radius ratio for the solution to the discrete problem, obtained in the Proposition \ref{weak Aleks bound}, significantly depends on the structure of this discretization. 

\begin{conj}\label{GIP conj}
Suppose $\mu$ and $\lambda$ are Borel measures on $\sn$, where $\lambda$ is absolutely continuous. If $\mu$ is not concentrated on a closed hemisphere and is weakly Aleksandrov related to $\lambda$, then there exists a convex body $K$ containing the origin in its interior, such that $\mu=\lambda(K,\cdot)$.
\end{conj}


\frenchspacing
\bibliographystyle{cpam}

\end{document}